\documentclass[onefignum,onetabnum]{siamart190516}

\usepackage{lipsum}
\usepackage{amsfonts}
\usepackage{graphicx}
\usepackage{epstopdf}
\usepackage{algorithmic}
\ifpdf
  \DeclareGraphicsExtensions{.eps,.pdf,.png,.jpg}
\else
  \DeclareGraphicsExtensions{.eps}
\fi
\usepackage{amsmath,bbm,cleveref}
\usepackage{subcaption, tikz, graphicx, pgfplots}
\graphicspath{{figures/}}
\usepackage{scalerel}

\newcommand{\de}{\: \mathrm{d}}
\newcommand{\ddt}[2]{\frac{\mathrm{d}^2#1}{\mathrm{d}#2^2}}

\newcommand{\ddp}[2]{\frac{\partial#1}{\partial#2}}
\newcommand{\R}{\mathbb{R}}

\newcommand{\Z}{\mathbb{Z}}
\newcommand{\C}{\mathbb{C}}
\renewcommand{\i}{\mathrm{i}}
\newcommand{\T}{\mathcal{T}}

\newsiamremark{remark}{Remark}
\newsiamremark{hypothesis}{Hypothesis}
\crefname{hypothesis}{Hypothesis}{Hypotheses}
\newsiamthm{claim}{Claim}

\headers{Asymptotic characterisation of localised defect modes}{R. V. Craster and B. Davies}

\title{Asymptotic characterisation of localised defect modes: Su-Schrieffer-Heeger and related models\thanks{Submitted to the editors 21 Jul 2022.
\funding{This work was funded by the H2020 FETOpen project BOHEME under grant agreement No.~863179.}}}

\author{Richard V. Craster\thanks{Department of Mathematics, Imperial College London, South Kensington Campus, London SW7~2AZ
  (\email{bryn.davies@imperial.ac.uk}).}
\and Bryn Davies\footnotemark[2]}

\usepackage{amsopn}

\ifpdf
\hypersetup{
  pdftitle={Asymptotic characterisation of localised defect modes},
  pdfauthor={R. V. Craster and B. Davies}
}
\fi

\begin{document}

\maketitle

% REQUIRED
\begin{abstract}
Motivated by topologically protected states in wave physics, we study localised eigenmodes in one-dimensional periodic media with defects. The Su-Schrieffer-Heeger model (the canonical example of a one-dimensional system with topologically protected localised defect states) is used to demonstrate the method. Our approach can be used to describe two broad classes of perturbations to periodic differential problems: those caused by inserting a finite-sized piece of arbitrary material and those caused by creating an interface between two different periodic media. The results presented here characterise the existence of localised eigenmodes in each case and, when they exist, determine their eigenfrequencies and provide concise analytic results that quantify the  decay rate of these modes. These results are obtained using both high-frequency homogenisation and transfer matrix analysis, with good agreement between the two methods. 
\end{abstract}

% REQUIRED

% REQUIRED

\section{Introduction}
A powerful tool for manipulating waves in periodic media is the principle that if a defect is introduced to an otherwise crystalline structure, then it is possible to create eigenmodes that are strongly localised in a neighbourhood of the defect \cite{bayindir2000tight, figotin1997localized, meade1991photonic}. These localised eigenmodes (often known as \emph{defect modes}) have eigenfrequencies that are within a stop band of the spectrum of the unperturbed periodic structure and typically decay exponentially fast away from the defect.  This phenomenon is similar to the localised states that exist in the neighbourhood of an impurity in a semiconductor \cite{yablonovitch1991donor} and can be used to build powerful waveguides, capable of guiding waves of specific frequencies both along straight lines and around corners \cite{khanikaev2013photonic, makwana2018geometrically, rechtsman2013photonic}.

A significant breakthrough in this field has been the development of \emph{topologically protected} localised eigenmodes \cite{khanikaev2013photonic, rechtsman2013photonic}. Many traditional defect modes, often created by local compact perturbations of a periodic medium, suffer from being very sensitive to the material's properties and their strong localisation quickly breaks down when random imperfections are introduced (see \emph{e.g.} Figure~9 of \cite{ammari2020topologically}). Topologically protected modes, however, exploit the topological properties of the underlying periodic structure to create localised defect modes that are robust with respect to random imperfections \cite{asboth2016short,  khanikaev2013photonic}. This significant development has been inspired by the creation of topological insulators \cite{asboth2016short, kane2005topological} and has greatly increased waveguides' efficacy by giving them enhanced robustness. 

Media supporting localised defect modes have previously been studied using various asymptotic methods. Defects where the periodic medium is only changed locally can be studied concisely using approaches based on the method of fictitious sources \cite{ammari2021functional, thompson2008interaction, wilcox2005modeling}. These studies have been based, for example, on multiple scattering formulations \cite{thompson2008interaction} and asymptotic expansions in terms of material parameters \cite{ammari2021functional}. There is also a growing body of theory providing rigorous foundations for the existence of topologically protected defect modes in terms of multi-scale expansions \cite{drouot2020defect, fefferman2017topologically, lin2021mathematical}.

In this work, we study one-dimensional models and develop an approach for describing localised eigenmodes induced by a broad class of perturbations. In particular, two types of defects are considered in this work: the first are perturbations caused by inserting a finite piece of one material into another periodic material, while the second is those caused by sticking together two different periodic materials. In the latter case, it is required that the two materials have the same spectral band gaps, so that the localised mode can have an eigenfrequency lying in stop bands of both media. This is a setting that often occurs in the study of topological insulators, particularly when the interface is created by making different perturbations to a periodic medium on either side of the defect \cite{ammari2020topologically, makwana2018geometrically}. 

The examples we will use to demonstrate our method will be versions of the Su-Schrieffer-Heeger model \cite{su1979solitons, hasan2010colloquium}. Arising in the study of electron localisation in polymers, the Su-Schrieffer-Heeger model (often abbreviated to `SSH') is commonly cited as the simplest periodic structure with non-trivial topology and has led to many subsequent studies in settings including acoustics \cite{zheng2019observation}, photonics \cite{xie2018second}, elastic plates \cite{chaplain2020topological} and mechanical systems \cite{huber2016topological}. Much of this work has studied Hamiltonians under tight-binding approximations, but there have also been attempts to reconcile this theory in analogous differential systems. A spectral analysis of a three-dimensional Helmholtz formulation of this problem was conducted in \cite{ammari2020robust, ammari2020topologically} using boundary integral analysis, including showing the robustness of the eigenfrequency of the localised eigenmode. Here, we study one-dimensional versions of these problems and are able to extend the theory in this case, particularly by quantifying the rate at which the localised eigenmode decays away from the defect.

In this work, we use both transfer matrices and high-frequency homogenisation to characterise the modes. Transfer matrices are a popular and long-standing approach for modelling the propagation of waves in periodic media \cite{markovs2008wave, SANCHEZSOTO2012191}. The multi-scale homogenisation method used here operates at ``high-frequency'' in the sense that we describe localised eigenmodes as perturbations of standing wave solutions with non-zero frequency. This is in contrast to the low-frequency asymptotics that is classically employed in homogenisation theory. This method was introduced in \cite{craster2010high} and developed to discrete systems in \cite{craster2010lattice}. It is closely related to other prominent two-scale homogenisation methods \cite{allaire1992homogenization, nguetseng1989general} and $k\cdot p$ theory \cite{willatzen2009kp}. The advantage of the homogenised approach for our problems is that, while transfer matrices are a fundamentally one-dimensional approach, high-frequency homogenisation can be easily generalised to multi-dimensional structures \cite{craster2010high}. Our new approach extends previous work on high-frequency homogenisation of discrete systems with defects \cite{makwana2013localised} to continuous differential problems. It also extends previous homogenisation results for homogeneous background media \cite{marigo2017effective} and generalises results for local perturbations to periodic potentials \cite{hoefer2011defect}. The results presented here build on work characterising the existence of topologically protected edge modes \cite{drouot2020defect, fefferman2017topologically, lin2021mathematical} by providing a means to quantify their most important properties.

The paper is split into two parts, corresponding to the two classes of perturbations we consider: perturbations caused by inserting a finite section of material are considered in \Cref{sec:bounded}, while the juxtaposition of two different media is considered in \Cref{sec:half}. In each case, we start by presenting the general theory, first characterising the localised eigenmodes in terms of transfer matrices before computing approximations using high-frequency homogenisation. We then demonstrate the theory on some tangible examples, including several examples based on the Su-Schrieffer-Heeger model. Finally, in \Cref{sec:rainbow}, we demonstrate the value of this theory by using it to design a custom rainbow filtering device.

\section{Bounded defects in periodic media} \label{sec:bounded}

We will begin by studying the problem of creating a defect in a periodic medium by inserting a finite-sized piece of another material. A perturbation of this kind is typically compact, meaning it is only a local perturbation and does not affect the underlying spectral band structure. Two examples are sketched in \Cref{fig:mesh}. The theory developed here does not need the length of the defect region to be equal to the periodicity length (or an integer multiple of it).

\subsection{Problem formulation}

\begin{figure}
\centering
\includegraphics[width=0.8\linewidth]{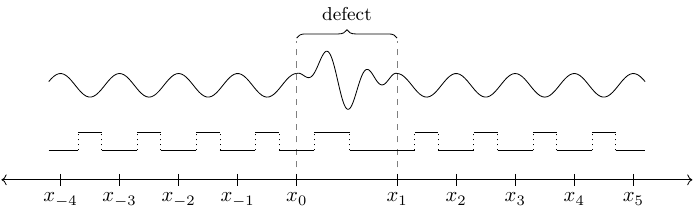}
\caption{We study transmission in a medium that is symmetric and periodic outside of some finite defect region. Two examples of suitable media are depicted above the coordinate mesh.} \label{fig:mesh}
\end{figure}

Let $h>0$ and define a mesh of points $\{x_n:n\in\Z\}$ which is such that
\begin{equation}
x_{n+1}-x_{n}=h \quad \text{for all } n\in\Z\setminus\{0\}.
\end{equation}
The distance $x_1-x_0$ need not be equal to $h$. Suppose that $c:\R\to\R$ is a piecewise smooth function which is strictly positive and is periodic on $\R\setminus(x_0,x_1)$, in the sense that
\begin{equation} \label{periodicity}
c(x_n+t)=c(x_m+t) \quad \text{for any } t\in[0,h) \text{ and any } n,m\in\Z\setminus\{0\}.
\end{equation}
Additionally, in order to simplify the analysis, we will assume that $c$ is symmetric in the regions where it is periodic, in the sense that
\begin{equation} \label{symmetry}
c(x_n+t)=c(x_{n+1}-t) \quad \text{for any } t\in[0,h) \text{ and any }n\in\Z\setminus\{0\}.
\end{equation}
This symmetry assumption is not essential, and our analysis can be generalised easily, but it will make the formulas much more concise and illuminating, while still allowing us to explore the interesting topological examples, based on the Su-Schrieffer-Heeger model. Within the defect region $(x_0,x_1)$, $c$ can be any strictly positive and piecewise smooth function. An example of the mesh  $\{x_n:n\in\Z\}$ is shown in \Cref{fig:mesh}, along with two suitable suitable functions.

We are interested in finding solutions $(u,\omega)\in C^2(\R^2)\times \C$ such that $u\not\equiv0$ and
\begin{equation} \label{eq:helmholtz}
\ddt{u}{x}+ \frac{\omega^2}{c^2(x)}u=0.
\end{equation}
We are interested in two types of solutions: propagating modes, which are elements of the space $L^2_\alpha(\mathbb{R}):=\{u\in L^2_{loc}(\mathbb{R}): u(x+h)=e^{\i \alpha h}u(x)\}$, and localised modes, which decay and are elements of $L^2(\mathbb{R})$. We will refer to the set of $\omega$ such that $(u,\omega)$ is a solution to \eqref{eq:helmholtz} as the spectrum associated to the geometry described by $c(x)$ and will denote it by $\sigma(c)$.

\begin{lemma}
The spectrum $\sigma(c)$ associated to any strictly positive $c(x)$ contains only real values.
\end{lemma}

\subsection{Periodic media}

We begin by exploring the case of a periodic medium with no defect. This corresponds to $x_0=x_1$ in our setting and we will use the notation $c_0$ to denote the fact that there is no defect region in $c$. We define the transfer matrix $T(\omega)\in\R^{2\times2}$, for any $\omega\in\R$, to be such that any solution $u$ to \eqref{eq:helmholtz} satisfies
\begin{equation} \label{eq:transfer_left}
\begin{pmatrix} u(x_{n+1}) \\ u'(x_{n+1}) \end{pmatrix} =
T(\omega) \begin{pmatrix} u(x_n) \\ u'(x_n) \end{pmatrix},
\end{equation}
for $n\in\Z\setminus\{0\}$, $\omega\in\R$. Under our assumption of symmetry \eqref{symmetry}, we also have that
\begin{equation} \label{eq:transfer_right}
\begin{pmatrix} u(x_{n-1}) \\ u'(x_{n-1}) \end{pmatrix} =
ST(\omega)S \begin{pmatrix} u(x_n) \\ u'(x_n) \end{pmatrix}, \qquad\text{where } S=\begin{pmatrix} 1 & 0 \\ 0 & -1 \end{pmatrix},
\end{equation}
for $n\in\Z\setminus\{0\}$, $\omega\in\R$.

\begin{lemma}
For any periodic medium satisfying \eqref{periodicity} and \eqref{symmetry}, the corresponding transfer matrix satisfies $\det(T(\omega))=1$ for all $\omega\in\R$.
\end{lemma}
\begin{proof}
From the composition of \eqref{eq:transfer_left} and \eqref{eq:transfer_right} we have that $T(\omega)ST(\omega)S=I$, from which the result follows.
\end{proof}

For periodic media $c_0$, the spectrum $\sigma(c_0)$ can be found using Floquet--Bloch theory. That is, any element of $\sigma(c_0)$ will also be in the spectrum of \eqref{eq:helmholtz} posed on the elementary cell $[0,h)$ with quasiperiodic boundary conditions: $u(0)=\exp(\i h \alpha)u(h)$ and $u'(0)=\exp(\i h \alpha)u'(h)$, where $\alpha\in\R$ is the Bloch wavenumber. This leads to the problem
\begin{equation}
\Big(T(\omega_\alpha)-e^{\i h \alpha}I\Big)\begin{pmatrix}u(x_n)\\u'(x_n)\end{pmatrix}=0.
\end{equation}
From which, we can see that $\omega_\alpha$ and $\alpha$ are related by the dispersion relation given by $\det(T(\omega_\alpha)-e^{\i h \alpha}I)=0$. The spectrum $\sigma(c_0)$ is then given by the union of all such $\omega_\alpha$, for $\alpha\in[0,2\pi/h)$. This theory will be useful for finding the spectrum of the background periodic medium. We are interested in studying localised eigenmodes whose eigenfrequencies satisfy $\omega\in\R\setminus\sigma(c_0)$. With this in mind, the decay of localised eigenmodes can be foreseen by the following result about $\omega\in\R\setminus\sigma(c_0)$:

\begin{lemma}
If $\omega\in\R\setminus\sigma(c_0)$, then $T(\omega)$ has real-valued eigenvalues satisfying $|\lambda_1|<1$ and $|\lambda_2|>1$.
\end{lemma}
\begin{proof}
Since $T(\omega)$ has $\det(T(\omega))=1$, it must be the case that its eigenvalues are either a complex conjugate pair with $|\lambda_1|=|\lambda_2|=1$ or they are real and satisfy $|\lambda_1|<1$ and $|\lambda_2|>1$. In the former case, there exists some $\alpha\in\R$ such that $\lambda_1=\exp(\i h \alpha)$ so it must hold that $\det(T(\omega)-e^{\i h \alpha}I)=0$, contradicting the fact that $\omega\notin\sigma(c_0)$.
\end{proof}

\subsection{Defect modes}

Consider a medium with some defect in the region $(x_0,x_1)$, where $x_1>x_0$. Define the matrix $D^L\in\R^{2\times2}$ in terms of the solution of two initial value problems on the defect region $[x_0,x_1]$:
\begin{align} \label{eq:defect_matrix}
&\text{Given} \quad \begin{cases}
\left(\ddt{}{x}+ \frac{\lambda^2}{c^2(x)}\right)\psi_{1,\lambda}=0, \\
\psi_{1,\lambda}(x_0) = 1, \\
\psi_{1,\lambda}'(x_0) = 0,
\end{cases}
\,\text{and}\quad
\begin{cases}
\left(\ddt{}{x}+ \frac{\lambda^2}{c^2(x)}\right)\psi_{2,\lambda}=0, \\
\psi_{2,\lambda}(x_0) = 0, \\
\psi_{2,\lambda}'(x_0) = 1,
\end{cases} \nonumber
\\
&\qquad\text{we define}\quad
D^L:=\begin{pmatrix} \psi_{1,\lambda}(x_1) & \psi_{2,\lambda}(x_1) \\ \psi_{1,\lambda}'(x_1) & \psi_{2,\lambda}'(x_1) \end{pmatrix}.
\end{align}
The matrix $D^L$ is the transfer matrix associated to the material in the defect region. We have that
\begin{equation} \label{eq:transfer_defect}
\begin{pmatrix} u(x_{n+1}) \\ u'(x_{n+1}) \end{pmatrix} =
T \begin{pmatrix} u(x_n) \\ u'(x_n) \end{pmatrix} + \delta_{n,0} (D^L-T) \begin{pmatrix} u(x_n) \\ u'(x_n) \end{pmatrix},
\end{equation}
for all $n\in\Z$. It is a standard result in transfer matrix theory to characterise the eigenfrequencies of localised defect eigenmodes using the defect matrix $D^L(\omega)$ and the eigenvectors of $T(\omega)$, see \emph{e.g.} \cite{schoenberg1983properties}. We present a version of this result, for completeness:

\begin{theorem} \label{thm:localised}
The eigenfrequency $\omega$ of a localised eigenmode of the Helmholtz problem \eqref{eq:helmholtz} posed on a medium with a bounded defect must satisfy
\begin{equation*}
\begin{pmatrix} -V_{21}(\omega) & V_{11}(\omega) \end{pmatrix}D^L(\omega)\begin{pmatrix} V_{11}(\omega) \\ -V_{21}(\omega) \end{pmatrix} =0,
\end{equation*}
where $D^L(\omega)$ is the defect matrix, defined in \eqref{eq:defect_matrix}, and $(V_{11}(\omega), V_{21}(\omega))^\top$ is the eigenvector of $T(\omega)$ associated to the eigenvalue $|\lambda_1|<1$. Furthermore, the localised eigenmode satisfies $V_{21}u(x_0)=-V_{11}u'(x_0)$.
\end{theorem}
\begin{proof}
We start with the values of the solution and its derivative at $x_0$, which we denote by $U_0$ and $U_0'$, and then use \eqref{eq:transfer_defect} to find a solution that decays in the far field in either direction. We use the methods of \cite{schoenberg1983properties} to handle the semi-infinite structures. We have that
\begin{equation} \label{eq:eigenval_right}
\begin{pmatrix} u(x_n) \\ u'(x_n) \end{pmatrix} = T^{n-1}D^L \begin{pmatrix} U_0 \\ U'_0 \end{pmatrix}=V\begin{pmatrix} \lambda_1^{n-1} & 0 \\ 0 & \lambda_2^{n-1}\end{pmatrix}
V^{-1}D^L \begin{pmatrix} U_0 \\ U'_0 \end{pmatrix},
\end{equation}
for $n\in\Z^{>0}$, and similarly
\begin{equation} \label{eq:eigenval_left}
\begin{pmatrix} u(x_{n}) \\ u'(x_{n}) \end{pmatrix} = ST^{-n}S \begin{pmatrix} U_0 \\ U'_0 \end{pmatrix}=SV\begin{pmatrix} \lambda_1^{-n} & 0 \\ 0 & \lambda_2^{-n}\end{pmatrix}V^{-1}S\begin{pmatrix} U_0 \\ U'_0 \end{pmatrix},
\end{equation}
for $n\in\Z^{\leq 0}$, where $|\lambda_1(\omega)|<1$ and $|\lambda_2(\omega)|>1$ are the eigenvalues of $T(\omega)$ and $V(\omega)$ is the matrix of associated eigenvectors. We are looking for localised eigenmodes which satisfy $\lim_{n\to\pm\infty}u(x_n)=0$ and $\lim_{n\to\pm\infty}u'(x_n)=0$. Since $\lim_{n\to\pm\infty}\lambda_1(\omega)^n=0$ but $\lim_{n\to\pm\infty}\lambda_2(\omega)^n\neq0$, we want to find $\omega$, $U_0$ and $U_0'$ such that
\begin{equation}
\begin{pmatrix} -V_{21} & V_{11} \end{pmatrix}D^L\begin{pmatrix} U_0 \\ U'_0 \end{pmatrix}=0
\quad\text{and}\quad
\begin{pmatrix} -V_{21} & V_{11} \end{pmatrix}S\begin{pmatrix} U_0 \\ U'_0 \end{pmatrix}=0.
\end{equation}
The second of these equations gives us that $(U_0, U_0')$ is proportional to $(V_{11}, -V_{21})$, which we can substitute into the first equation to obtain the desired formula.
\end{proof}

As well as identifying the eigenfrequency of localised eignmodes, which is determined by the formula in \Cref{thm:localised}, we want to quantify the rate at which the mode decays away from the defect. Since \Cref{thm:localised} gives the values for $u(x_0)$ and $u'(x_0)$, we can use the transfer matrix $T(\omega)$ to propagate the solution any integer number of steps away from the defect.

\begin{corollary} \label{cor:eigenval_decay}
A localised eigenmode $u$ of \eqref{eq:helmholtz}, posed on a medium with a bounded defect, and its associated eigenfrequency $\omega$ must satisfy
\begin{equation*}
u(x_n)=O\Big(|\lambda_1(\omega)|^{|n|}\Big) \quad \text{and}\quad u'(x_n)=O\Big(|\lambda_1(\omega)|^{|n|}\Big) \quad\text{as }n\to\pm\infty,
\end{equation*}
where $\lambda_1(\omega)$ is the eigenvalue of $T(\omega)$ satisfying $|\lambda_1(\omega)|<1$.
\end{corollary}
\begin{proof}
From \Cref{thm:localised} we can see that $(u(x_0), u'(x_0))^\top$ is proportional to $(V_{11}, -V_{21})^\top$. Using the fact that $D^L(V_{11},-V_{21})^\top$ is proportional to $(V_{11},V_{21})^\top$, the equations \eqref{eq:eigenval_left} and \eqref{eq:eigenval_right} can be rewritten as
\begin{equation}
\begin{pmatrix} u(x_n) \\ u'(x_n) \end{pmatrix} = \lambda_1(\omega)^{n-1}\begin{pmatrix} V_{11} \\ V_{21} \end{pmatrix}  \qquad \text{for } n\in\Z^{>0},
\end{equation}
and similarly
\begin{equation}
\begin{pmatrix} u(x_{n}) \\ u'(x_{n}) \end{pmatrix}=\lambda_1(\omega)^{-n}S\begin{pmatrix} V_{11} \\ V_{21} \end{pmatrix}  \qquad \text{for } n\in\Z^{\leq 0},
\end{equation}
from which the result follows.
\end{proof}

\subsection{Homogenisation of localised defect modes} \label{sec:HFH}

The rate at which localised eigenmodes decay away from the defect region can be bounded using the eigenvalues of transfer matrices, as shown in \Cref{cor:eigenval_decay}. We will now use high-frequency homogenisation to build on this by computing an asymptotic approximation for the envelope of the localised eigenmode. As well as verifying the sharpness of the bound from \Cref{cor:eigenval_decay}, it will establish an approach that can be developed to multi-dimensional systems. 

Since the defect region is not assumed to be symmetric, we start by formulating a version of \eqref{eq:transfer_defect} from the other direction. We define the matrix $D^R\in\mathbb{R}^{2\times2}$ in terms of the solutions of the two initial value problems on $(x_0,x_1]$:
\begin{align} \label{eq:defect_matrix_right}
&\text{Given}\quad
\begin{cases}
\left(\ddt{}{x}+ \frac{\lambda^2}{c^2(x)}\right)\psi_{1,\lambda}=0, \\
\psi_{1,\lambda}(x_1) = 1, \\
\psi_{1,\lambda}'(x_1) = 0,
\end{cases}
\,\text{and}\quad
\begin{cases}
\left(\ddt{}{x}+ \frac{\lambda^2}{c^2(x)}\right)\psi_{2,\lambda}=0, \\
\psi_{2,\lambda}(x_1) = 0, \\
\psi_{2,\lambda}'(x_1) = 1,
\end{cases}
\nonumber \\
&\qquad\text{we define}\quad
D^R:=\begin{pmatrix} \psi_{1,\lambda}(x_0) & \psi_{2,\lambda}(x_0) \\ \psi_{1,\lambda}'(x_0) & \psi_{2,\lambda}'(x_0) \end{pmatrix}
\end{align}
and then we have that
\begin{equation} \label{eq:transfer_defect_right}
\begin{pmatrix} u(x_{n-1}) \\ u'(x_{n-1}) \end{pmatrix} =
STS \begin{pmatrix} u(x_n) \\ u'(x_n) \end{pmatrix} + \delta_{n,1} (D^R-STS) \begin{pmatrix} u(x_n) \\ u'(x_n) \end{pmatrix},
\end{equation}
for all $n\in\Z$.

Adding \eqref{eq:transfer_defect} and \eqref{eq:transfer_defect_right}, we find that
\begin{equation} \label{eq:discrete}
\begin{split}
\begin{pmatrix} u(x_{n+1}) + u(x_{n-1}) \\ u'(x_{n+1})-u'(x_{n-1}) \end{pmatrix} =
2T \begin{pmatrix} u(x_n) \\ 0 \end{pmatrix} &+ \delta_{n,0} (D^L-T)\begin{pmatrix} u(x_n) \\ u'(x_n) \end{pmatrix}
\\&\quad+\delta_{n,1} (SD^RS-T) \begin{pmatrix} u(x_n) \\ -u'(x_n) \end{pmatrix}.
\end{split}
\end{equation}
The difference equation \eqref{eq:discrete} can be viewed as the discrete version of \eqref{eq:helmholtz} and will be studied using high-frequency homogenisation in a similar way to the approach used in \cite{craster2010lattice, makwana2013localised} for the study of discrete lattices. Within this discrete framework, it does not make sense to think of $u'(x_n)$ as the derivative of $u(x_n)$. Rather, we should think of $u'(x_n)$ as a discrete quantity that is related to $u(x_n)$ by \eqref{eq:discrete}. To avoid confusion with other derivatives, we will relabel $v(x_n)=u'(x_n)$ from here onwards.

We will apply a high-frequency homogenisation ansatz to \eqref{eq:discrete} in the case that the perturbation is small. In particular, we introduce a small parameter $0<\epsilon\ll1$ which captures the magnitude of the perturbation, in the sense that $\| D^L-T\|=O(\epsilon)$. In particular, we suppose that $\epsilon$ is such that we can write
\begin{equation}
\epsilon P^L = D^L-T \quad\text{and}\quad \epsilon P^R=SD^RS-T,
\end{equation}
for some matrices $P^L,P^R\in\R^{2\times2}$, whose entries are $O(1)$ as $\epsilon\to0$. In addition to the discrete variable, we introduce the continuous long-scale variable $\eta=\epsilon n$. This long scale $\eta$ will capture the growth and decay of eigenmodes, while they oscillate on the discrete short scale $n$. This short scale is discrete as the value of $u$ and its derivative at each mesh point fully determines the field within each periodic unit (and the discrete values can be propagated using transfer matrices). We follow \cite{craster2010lattice, makwana2013localised} and seek an asymptotic characterisation of localised eigenmodes as perturbations of standing waves (\emph{i.e.} the solutions at the edges of the band gap). The reference to ``high'' frequencies in the name of this approach connotes the fact that such waves exist at non-zero frequencies (which are ``high'' in comparison to the low-frequency asymptotic analysis typically performed in classical homogenisation theory). With this in mind, we express the short-scale variable relative to a single reference mass and set
\begin{equation} \label{hfh_ansatz1}
\begin{gathered}
u(x_{n-1})=u(\eta-\epsilon,-1)=-u(\eta-\epsilon,0), \qquad
u(x_{n})=u(\eta,0)=u(\eta,0),\\
u(x_{n+1})=u(\eta+\epsilon,1)=-u(\eta+\epsilon,0),
\end{gathered}
\end{equation}
where we have assumed anti-periodicity between adjacent mesh points, since we are searching for localised eigenmodes near to the edges of the bands in the Bloch dispersion diagram, where $\alpha=\pi$ in the examples considered here (this can be easily adapted for the periodic case). The choice to impose the anti-periodicity for $\eta\mapsto\eta+\epsilon$ in \eqref{hfh_ansatz1} means that $h=x_{n+1}-x_n=\epsilon$. This choice is somewhat arbitrary and could be modified to suit different relations between the perturbation size and the periodicity length (for example, \cite{craster2010lattice} treats a perturbation that is proportional to the square of the periodicity length). We adopt a similar ansatz for the derivative $v$. Inserting this multi-scale approach into the difference equation \eqref{eq:discrete} gives
\begin{equation} \label{eq:discrete_multiscale}
\begin{pmatrix} -2u-\epsilon^2 u_{\eta\eta} \\ -2\epsilon v_\eta \end{pmatrix} =
2T \begin{pmatrix} u \\ 0 \end{pmatrix} + \epsilon^2\delta(\eta) P^L\begin{pmatrix} u \\ v \end{pmatrix} + \epsilon^2\delta(\eta-\epsilon) P^R \begin{pmatrix} u \\ -v \end{pmatrix} +O(\epsilon^3),
\end{equation}
where $u=u(\eta,0)$, $v=v(\eta,0)$, subscripts are used to denote derivatives and where we have used $\delta_{n,0}=\epsilon\delta(\eta)$ to move the defects to the long scale.

To analyse the solutions to the multi-scale problem \eqref{eq:discrete_multiscale}, we adopt the ansatz
\begin{equation} \label{hfh_ansatz2}
u = u_0+\epsilon u_1 + \epsilon^2 u_2+\dots,  \qquad v = v_0+\epsilon v_1 + \epsilon^2 v_2+\dots.
\end{equation}
We similarly assume that the frequency $\omega$ of a localised eigenmode can be expanded in terms of $\epsilon$, where the leading-order term in the expansion is the value of $\omega$ at one of the edges of the band gap. This is consistent with the anti-periodicity condition that was asserted in \eqref{hfh_ansatz1}. Substituting this expansion into $T(\omega)$ gives an expansion of the form
\begin{equation}
T(\omega) = T_0 + \epsilon T_1 + \epsilon^2 T_2+\dots.
\end{equation}
Likewise, we have that $P^L(\omega)=P_0^L+O(\epsilon)$ and $P^R(\omega)=P_0^R+O(\epsilon)$ for some matrices $P_0^L$ and $P_0^R$. In each case, the leading-order term in the expansion is the value at the edge of the band gap. Substituting this into \eqref{eq:discrete_multiscale} gives a hierarchy of equations to be solved. At leading order, we find that $-2u_0=2(T_0)_{11}u_0$, so $(T_0)_{11}=-1$, corresponding to the fact that the solution is anti-periodic at leading order. We also conclude from this that $u_0$ depends only on the long-scale variable $\eta$, so we can write that
\begin{equation}
u_0(\eta,0)=f(\eta)
\end{equation}
for some function $f:\R\to\R$. This function $f(\eta)$ describes how the amplitude of the localised eigenmode varies on the long scale; it captures the \emph{envelope} of the eigenmode and solving for it is the main aim of the analysis in this section. The order-$\epsilon$ terms reveal that $-2u_1=2(T_0)_{11}u_1+2(T_1)_{11}u_0$, from which we conclude that $(T_1)_{11}=0$ and that $u_1$ is an arbitrary multiple of $u_0$, which we absorb into $u_0$ without any loss of generality. They also tell us that $-2(u_0')_\eta=(T_1)_{21}u_0$. Finally, the second-order terms reveal that
\begin{equation} \label{eq:general_f}
0=\ddt{f}{\eta}-\T_f^2 f+\delta(\eta) \left( (P_0^L)_{11} f + (P_0^L)_{12} v_0 \right) + \delta(\eta-\epsilon) \left( (P_0^R)_{11} f - (P_0^R)_{12} v_0 \right),
\end{equation}
where we have defined the constant $\T_f:=\sqrt{-2(T_2)_{11}}$. To solve equations of the form \eqref{eq:general_f}, we have the following lemma:

\begin{lemma} \label{lem:f_solution}
Let $k\in\R$, $\alpha\in(0,\infty)$ and $g:\R\to\R$ be some continuous function. Then, the function
\begin{equation*}
f(\eta) = \frac{g(k)}{2{\alpha}} e^{-{\alpha}|\eta-k|}, \qquad \eta\in\R,
\end{equation*}
is a solution to the second-order ordinary differential equation
\begin{equation*}
0=\ddt{f}{\eta}- \alpha^2 f(\eta)+ \delta(\eta-k)g(\eta),
\end{equation*}
which decays as $\eta\to\pm\infty$.
\end{lemma}

Using \Cref{lem:f_solution}, we see that \eqref{eq:general_f} has a localised solution given by
\begin{equation} \label{f_gen_solution}
f(\eta)=\tfrac{(P_0^L)_{11} f(0) + (P_0^L)_{12} v_0(0) }{2\T_f} e^{-\T_f|\eta|} + \tfrac{(P_0^R)_{11} f(\epsilon) - (P_0^R)_{12} v_0(\epsilon) }{2\T_f} e^{-\T_f|\eta-\epsilon|}.
\end{equation}

In the case that the defect is symmetric, we have that $P^L=P^R$ and expanding the terms in \eqref{f_gen_solution} about the point $\eta=\epsilon/2$ gives the following concise result:

\begin{theorem} \label{thm:HFH_envelope}
Suppose that a periodic, symmetric material has a bounded defect that is symmetric in the sense that $c(x_0+t)=c(x_1-t)$ for all $t\in(0,x_1-x_0)$. Then, under the homogenisation ansatz \eqref{hfh_ansatz2}, a localised eigenmode must satisfy
\begin{equation*}
|u(x_n)|=\exp\left(-\left((D_0^L)_{11}-(T_0)_{11}\right)|n-1/2|\right)+O(\epsilon) \qquad \text{for } n\in\Z,
\end{equation*}
when normalised such that $\max_n|u(x_n)|=1$.
\end{theorem}
\begin{proof}
Expanding \eqref{f_gen_solution} about the point $\eta=\epsilon/2$ and using the fact that $P_0^L=P_0^R$ gives that
\begin{equation}
f(\eta)=\frac{(P_0^L)_{11} f(\epsilon/2)}{\T_f} e^{-\T_f|\eta-\epsilon/2|}+O(\epsilon),
\end{equation}
from which we can see that $(P_0^L)_{11}=\T_f$. Swapping coordinates back to the discrete short-scale variable and substituting the definition of $P^L$ gives the result.
\end{proof}

In the case that the defect is not symmetric, it is slightly harder to find the constants in \eqref{f_gen_solution} as the coefficients of $v_0$ do not generally cancel each other out when the two terms are expanded about $\eta=\epsilon/2$. However, this can be overcome by repeating the above analysis to find the leading-order part of the derivative $v_0$. Subtracting \eqref{eq:transfer_defect_right} from \eqref{eq:transfer_defect} gives the difference equation
\begin{equation*} %\label{eq:discrete_v2}
\begin{pmatrix} u(x_{n+1}) - u(x_{n-1}) \\ v(x_{n+1})+v(x_{n-1}) \end{pmatrix} =
2T \begin{pmatrix} 0 \\ v(x_n) \end{pmatrix} + \delta_{n,0}\epsilon P^L\begin{pmatrix} u(x_n) \\ v(x_n) \end{pmatrix}
+\delta_{n,1} \epsilon P^R \begin{pmatrix} -u(x_n) \\ v(x_n) \end{pmatrix}.
\end{equation*}
Repeating the homogenisation method used above, we see that $(T_0)_{22}=-1$ and that we can write $v_0(\eta,0)=g(\eta)$ for some function $g:\R\to\R$. Continuing to higher orders, we find that $(T_1)_{22}=0$ and, subsequently, that
\begin{equation} \label{eq:general_g}
0=\ddt{g}{\eta}-\T_g^2 g+\delta(\eta) \left( (P_0^L)_{21} f + (P_0^L)_{22} g \right) + \delta(\eta-\epsilon) \left( -(P_0^R)_{21} f + (P_0^R)_{22} g \right),
\end{equation}
where $\T_g:=\sqrt{-2(T_2)_{22}}$. We use \Cref{lem:f_solution} to find that the localised solution is given by
\begin{equation} \label{g_gen_solution}
g(\eta)=\tfrac{(P_0^L)_{21} f(0) + (P_0^L)_{22} g(0) }{2\T_g} e^{-\T_g|\eta|} + \tfrac{-(P_0^R)_{21} f(\epsilon) + (P_0^R)_{22} g(\epsilon) }{2\T_g} e^{-\T_g|\eta-\epsilon|}.
\end{equation}
Finally, we can then use \eqref{f_gen_solution} and \eqref{g_gen_solution} to formulate a matrix eigenvalue problem for $(f(0),g(0),f(\epsilon),g(\epsilon))^\top$, which we can solve to find $\T_f$ and $\T_g$ in terms of $P_0^L$ and $P_0^R$.

\subsection{Example: local perturbation} \label{sec:local} 

The simplest examples which can be handled using the methods developed in this work are those of \emph{local defects} (often known as \emph{point defects} in discrete settings). In this case, the defect region has the same width as the repeating unit cell (\emph{i.e.} $x_1-x_0=h$). This means that the defect medium can be obtained from the periodic medium by only changing the material within defect region. This is a \emph{compact} perturbation which preserves the underlying band structure.

To demonstrate the method, we study an example of a locally perturbed piecewise constant medium, given by
\begin{equation} \label{eq:c_defect}
c_{R-r}(x) = \begin{cases}
\frac{1}{r} & 2n+\frac{1}{2}\leq x<2n+\frac{3}{2}, n\in\Z\setminus\{0\}, \\
\frac{1}{R} & \frac{1}{2}\leq x<\frac{3}{2}, \\
1 & 2n-\frac{1}{2}\leq x<2n+\frac{1}{2}, n\in\Z,
\end{cases}
\end{equation}
for some positive constants $r$ and $R$. This is depicted in \Cref{fig:local}. This was studied previously by \cite{figotin1998localized, figotin1997localized} and can be viewed as a continuous version of the discrete model studied in \cite{makwana2013localised} and as a one-dimensional version of the two-dimensional model studied using boundary integral analysis in \cite{ammari2021functional}. Further to these existing results, we can use our methods to calculate the (leading-order) envelope of the localised mode.

\begin{figure}
\centering
\includegraphics[width=0.8\linewidth]{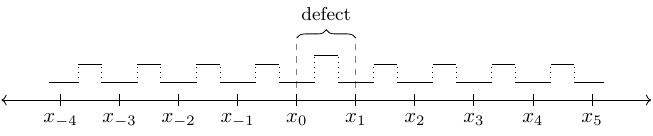}
\caption{A symmetric and piecewise constant periodic medium which is locally perturbed within the defect region, as defined in \eqref{eq:c_defect}, supports localised eigenmodes which can be described using this theory.} \label{fig:local}
\end{figure}

\begin{figure}
\centering
\begin{subfigure}[b]{0.7\linewidth}
\includegraphics[width=\linewidth]{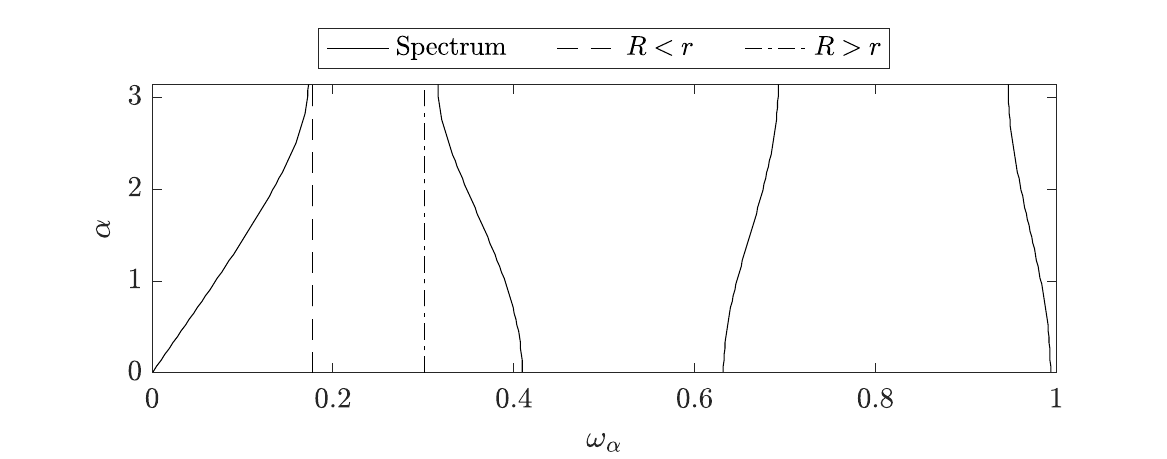}
\caption{} \label{PDspec}
\end{subfigure}

\begin{subfigure}[b]{0.43\linewidth}
\includegraphics[width=\linewidth]{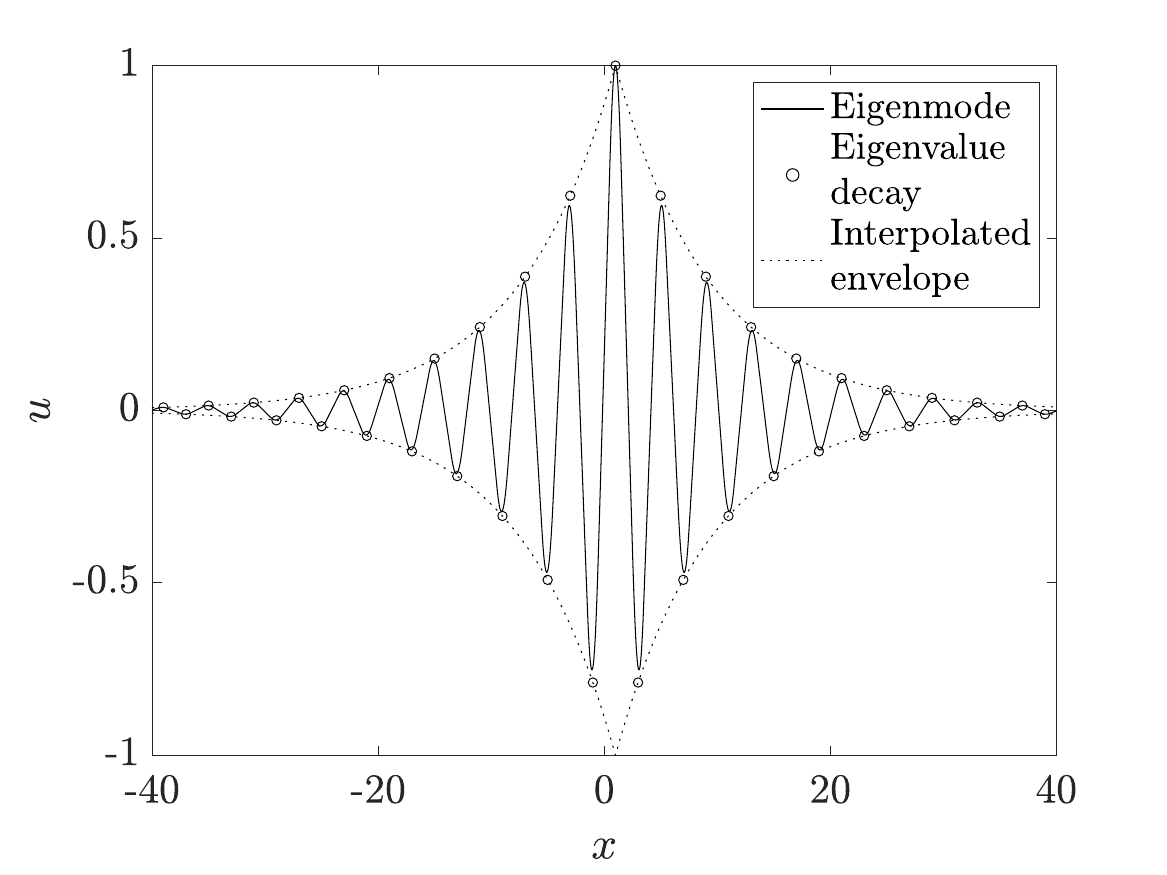}
\caption{} \label{TM_Rless}
\end{subfigure}
\begin{subfigure}[b]{0.43\linewidth}
\includegraphics[width=\linewidth]{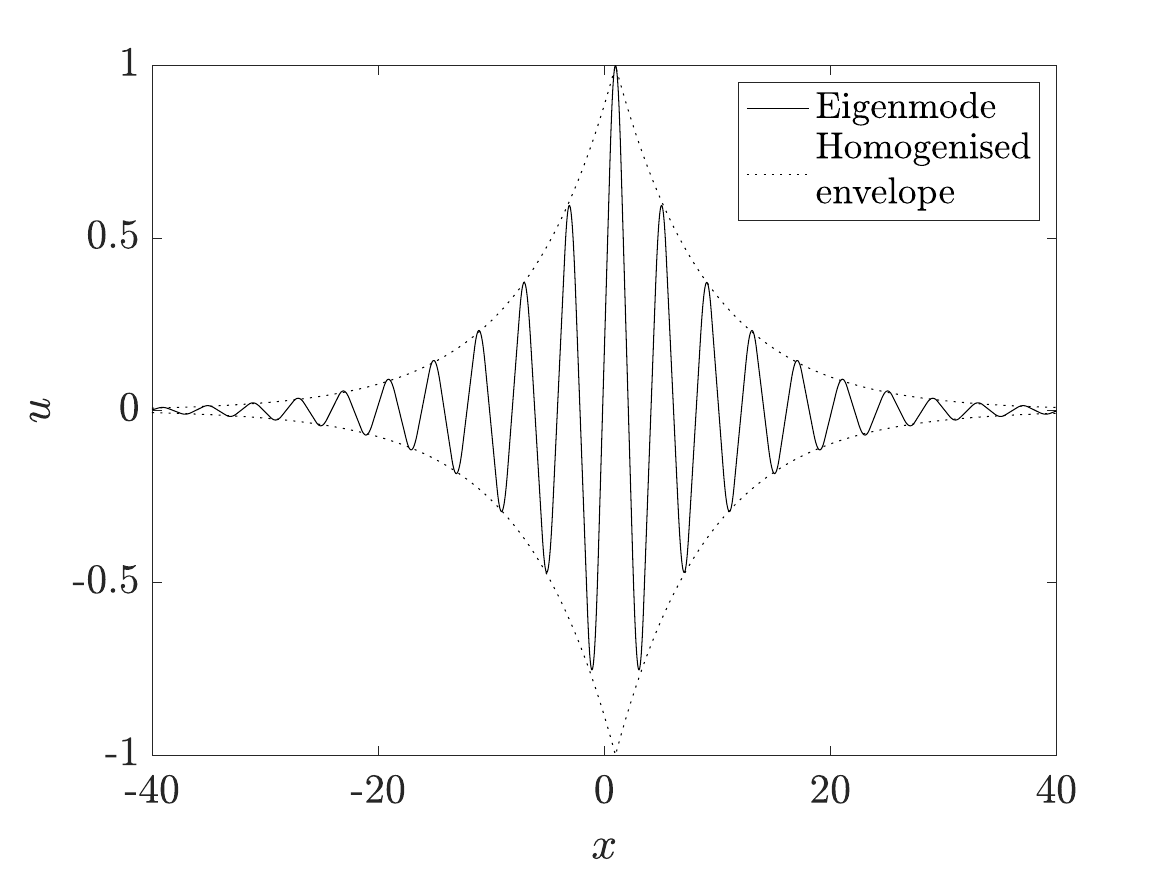}
\caption{} \label{HFH_Rless}
\end{subfigure}

\begin{subfigure}[b]{0.43\linewidth}
\includegraphics[width=\linewidth]{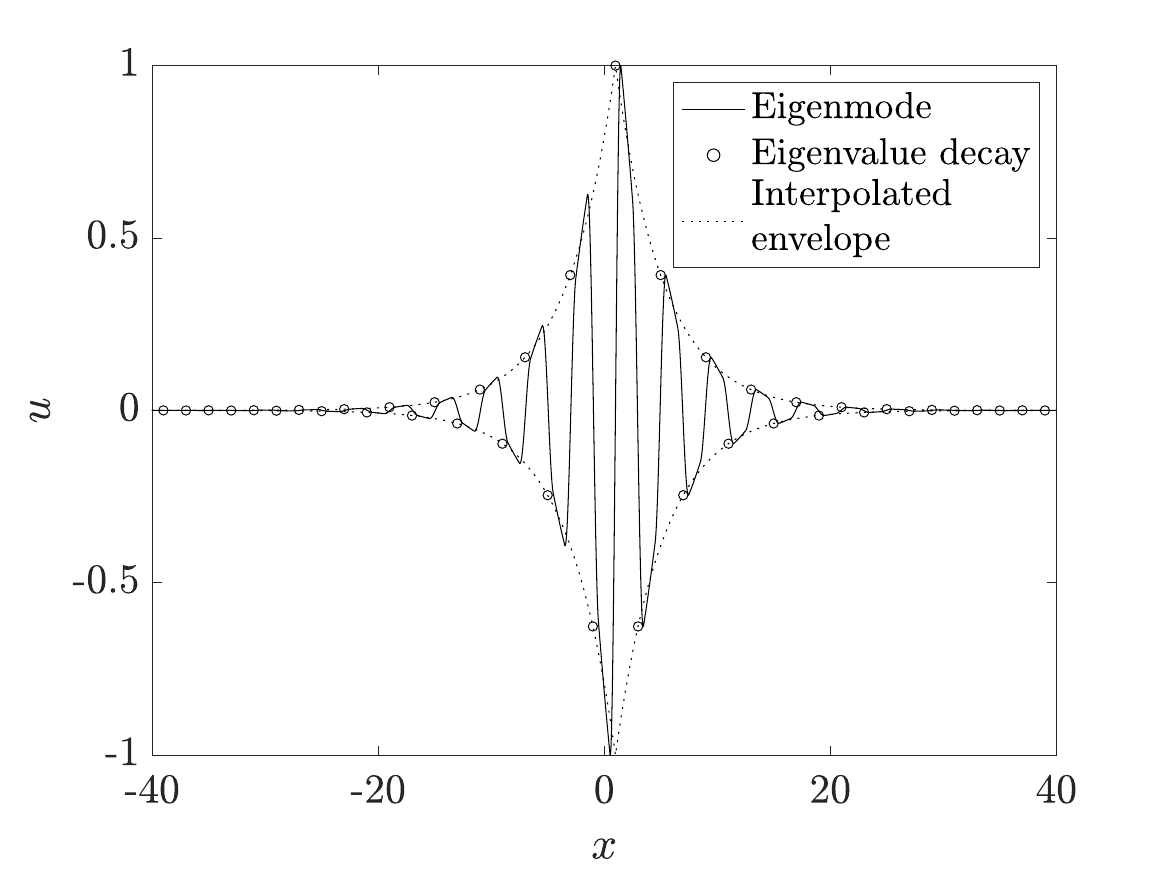}
\caption{} \label{TM_Rmore}
\end{subfigure}
\begin{subfigure}[b]{0.43\linewidth}
\includegraphics[width=\linewidth]{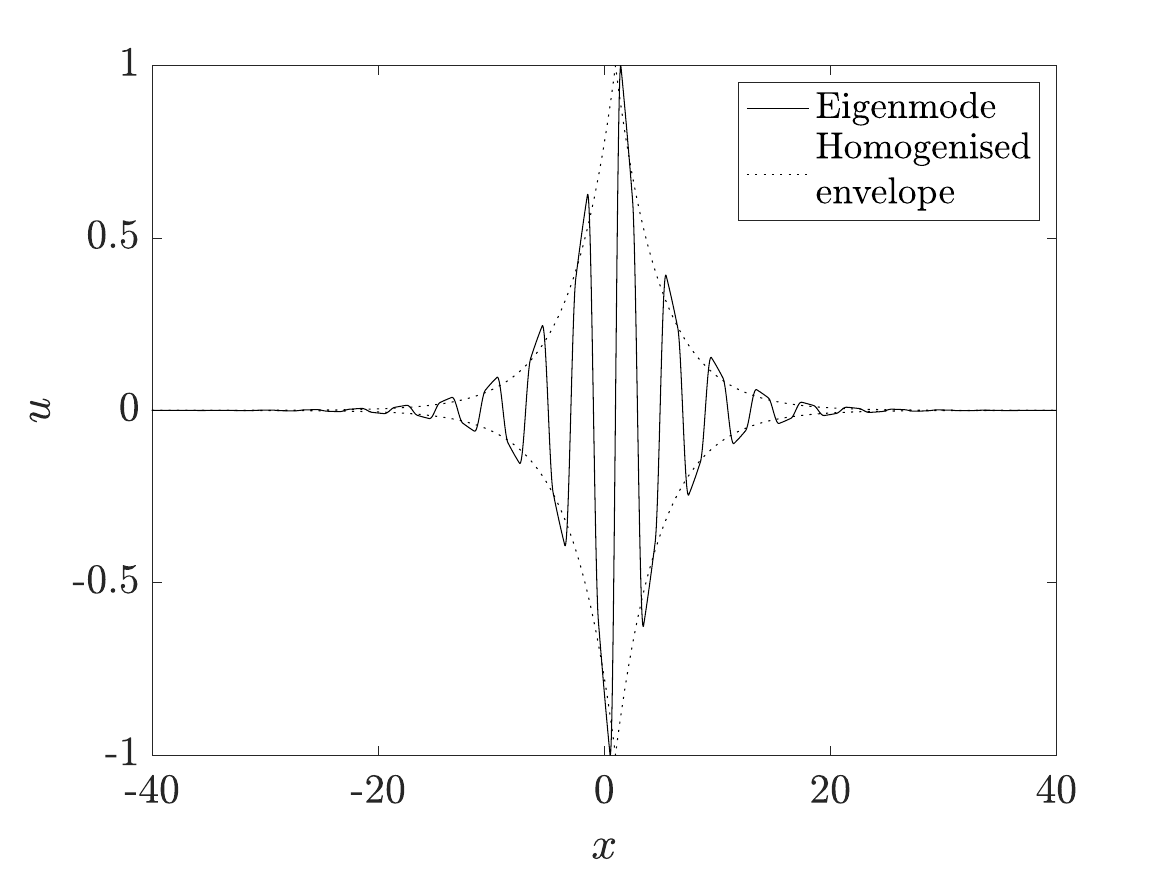}
\caption{} \label{HFH_Rmore}
\end{subfigure}
\caption{The localised eigenmodes of the locally perturbed medium \eqref{eq:c_defect}. The spectrum of the unperturbed periodic medium is shown in (a), along with the frequencies of the localised eigenmodes that exist for $R<r$ and $R>r$. The localised eigenmodes are shown in (b) and (c) for $R<r$ and in (d) and (e) for $R>r$. In (b) and (d), envelope functions are computed using transfer matrix eigenvalues (\Cref{cor:eigenval_decay}) and the homogenised envelopes are shown in (c) and (e). Here, $R=r\pm0.1$ and $r=10$.} \label{fig:PointDefectModes}
\end{figure}

We can compute the transfer matrix explicitly as
\begin{equation} \label{eq:T_PointDefect}
    T(\omega) = T_\text{c}(1,\tfrac{1}{2},\omega)\, T_\text{c}(r,1,\omega)\, T_\text{c}(1,\tfrac{1}{2},\omega),
\end{equation}
where the matrix $T_\text{c}\in\R^{2\times2}$ is used to describe each piecewise constant section and is given by
\begin{equation} \label{T_piecewise}
    T_\text{c}(r,l,\omega):=\begin{pmatrix} \cos(r l \omega) & \frac{1}{r\omega}\sin(r l \omega) \\ -r\omega\sin(r l \omega) & \cos(r l \omega) \end{pmatrix}.
\end{equation}
This gives that
\begin{align*}
T = {\small\begin{pmatrix}
\cos(\omega)\cos(\omega r)-\frac{r^2+1}{2r}\sin(\omega)\sin(\omega r) &
T_{12}(\omega) \\
-\omega^2 T_{12}(\omega)+\omega\frac{1-r^2}{r}\sin(\omega r) &
\cos(\omega)\cos(\omega r)-\frac{r^2+1}{2r}\sin(\omega)\sin(\omega r)
\end{pmatrix}},\\
\text{where  } T_{12}(\omega)=\frac{1}{\omega}\left( \sin(\omega)\cos(\omega r) + \left( \frac{1-r^2}{2r}+\frac{1+r^2}{2r}\cos(\omega) \right)\sin(\omega r) \right). \nonumber
\end{align*}
The defect matrix $D^L$ can be found by replacing $r$ with $R$ in the expression for $T$. With this in hand, we use \eqref{thm:localised} to find the frequencies of localised eigenmodes when $R=r\pm\epsilon$ for some $0<\epsilon\ll1$. This can be done explicitly, thanks to \eqref{eq:T_PointDefect}, however the expressions are too arduous to include here. The outcome of solving the equation derived in \Cref{thm:localised} numerically is shown in \Cref{PDspec}. We see that localised eigenmodes exist in both the $R<r$ and $R>r$ cases. When $R<r$, the localised eigenmode is a perturbation of the first mode of the system and is shown in \Cref{TM_Rless,HFH_Rless}. When $R<r$, the localised eigenmode is a perturbation of the second mode of the system and is shown in \Cref{TM_Rmore,HFH_Rmore}. The reason for the different frequency values for the localised eigenmodes in the two cases is the requirement from \Cref{thm:HFH_envelope} that $(D_0^L)_{11}(\omega)-(T_0)_{11}(\omega)>0$. We have that
\begin{equation}
\begin{split}
    (D_0^L)_{11}(\omega)-(T_0)_{11}(\omega)&=\cos(\omega)\cos(\omega R)-\frac{R^2+1}{2R}\sin(\omega)\sin(\omega R)\\&\qquad-\cos(\omega)\cos(\omega r)+\frac{r^2+1}{2r}\sin(\omega)\sin(\omega r).
\end{split}
\end{equation}
In \Cref{fig:PointDefect_P} we show how $(D_0^L)_{11}(\omega)-(T_0)_{11}(\omega)$ varies as a function of $\omega$ in the band gap of the periodic medium. We see that if $R<r$ then $(D_0^L)_{11}(\omega)-(T_0)_{11}(\omega)>0$ for $\omega$ at the lower end of the band gap, whereas the condition is satisfied at the upper end of the band gap if $R<r$.

\begin{figure}
    \centering
    \includegraphics[width=0.7\linewidth]{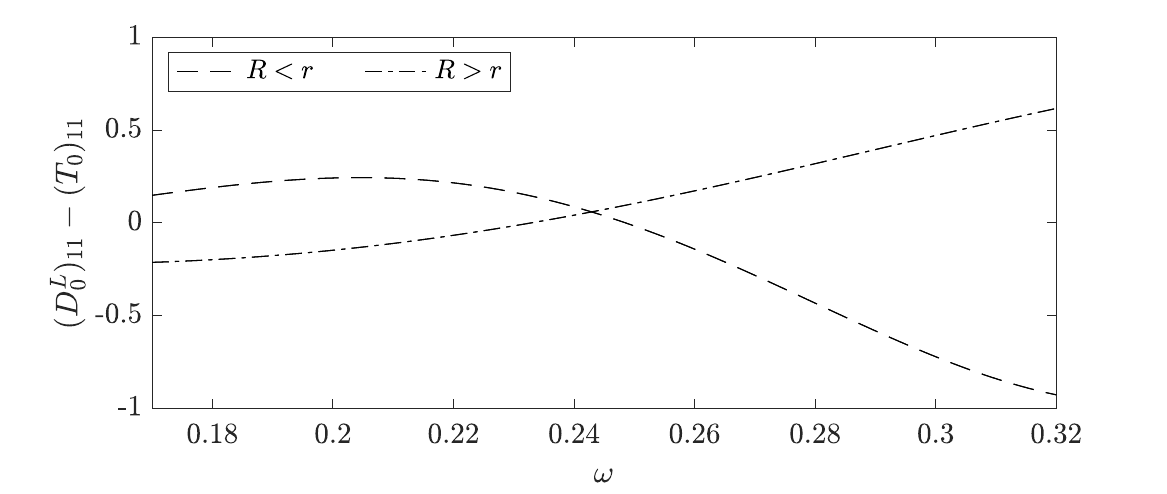}
    \caption{For a localised eigenmode to exist, the defect must be such that $(D_0^L)_{11}-(T_0)_{11}>0$ at the frequency of the defect mode. Here, we show how $(D_0^L)_{11}(\omega)-(T_0)_{11}(\omega)$ varies for $\omega$ in the first band gap of the locally perturbed medium \eqref{eq:c_defect} with $R=r\pm0.1$ and $r=10$.}
    \label{fig:PointDefect_P}
\end{figure}

We can use \Cref{cor:eigenval_decay} to use the eigenvalues of the transfer matrix $T$ to estimate the rate at which the localised eigenmodes decay away from the defect region. This is shown in \Cref{TM_Rless,TM_Rmore}, where the circles depict the decaying sequence of discrete values and the dotted line is an interpolated exponential envelope. Similarly, since the medium is symmetric within the defect region, we can use \Cref{thm:HFH_envelope} to estimate the envelope function using high-frequency homogenisation. This is shown in \Cref{HFH_Rless,HFH_Rmore}. There is good agreement between the two approaches, with the transfer matrix slightly underestimating the rate of decay, as expected: for $R<r$ the homogenised envelope gives a coefficient $\T_f=1.23$ and the interpolated transfer matrix eigenvalues give $\T_f=1.13$, when $\epsilon=0.1$. For $R>r$, the homogenised envelope gives a coefficient $\T_f=2.57$ and the interpolated transfer matrix eigenvalues give $\T_f=2.34$, again for $\epsilon=0.1$.

\subsection{Example: topological perturbation} \label{sec:noncompact}

In this section, we study a periodic medium that is composed of alternating short and long gaps between material inclusions. This is inspired by the Su-Schrieffer-Heeger model, which has alternating coupling parameters between atoms. The periodic medium is defined as
\begin{equation} \label{eq:c_SSH}
c_0(x) = \begin{cases}
\frac{1}{r} & x\in[n+\frac{1-d-a}{2},n+\frac{1-d+a}{2})\cup[n+\frac{1+d-a}{2},n+\frac{1+d+a}{2}), n\in\Z,\\
1 & \text{otherwise},
\end{cases}
\end{equation}
for some parameters $r\neq1$, $0<a<\frac{1}{2}$ and $a<d<1-a$. This medium is periodic and symmetric so can be handled using the theory developed here. Since it is piecewise constant, the transfer matrix $T$ can be written as
\begin{equation}
    T=T_\text{c}(1,\tfrac{1-d-a}{2},\omega) T_\text{c}(r,a,\omega) T_\text{c}(1,d-a,\omega) T_\text{c}(r,a,\omega) T_\text{c}(1,\tfrac{1-d-a}{2},\omega),
\end{equation}
where the matrix $T_\text{c}$ was defined in \eqref{T_piecewise}. When $d=1/2$ the first two bands stick together but for $d\neq1/2$ a band gap opens, as depicted in \Cref{fig:SSHbands}. 

It turns out that the band gap depicted in \Cref{fig:SSHbands} is \emph{topologically non-trivial}, in the sense that associated topological indices can take non-trivial values for certain values of $d$. In the context of these one-dimensional models, the natural topological quantity to consider is the Zak phase \cite{zak1989phase}, defined as
\begin{equation}
    \varphi_\mathrm{z} = \i\int_{-\pi/h}^{\pi/h} \langle u_\alpha , \ddp{}{\alpha} u_\alpha\rangle \de \alpha,
\end{equation}
where $\langle\cdot,\cdot\rangle$ is the $L^2([0,h])$ inner product and $u_\alpha\in L^2_\alpha(\R)$ is the solution to \eqref{eq:helmholtz} with the dependence on the Bloch parameter $\alpha$ added for emphasis. It can be shown for this example that $\varphi_\mathrm{z}=0$ when $d<1/2$ whereas $\varphi_\mathrm{z}=\pi$ when $d>1/2$. It is in this sense that the band gap that opens when $d>1/2$ is said to be topologically non-trivial.

The crucial principle of topological localised modes is that localised eigenmodes typically exist at interfaces between two materials that are associated to different values of a topological index. Extensive theory supporting this principle and proving the existence of these topological localised modes has been developed, particularly in one-dimensional systems, for example by \cite{drouot2020defect, fefferman2017topologically, lin2021mathematical}. In \Cref{sec:half}, we will study the typical Su-Schrieffer-Heeger interface problem, where a material with $d<1/2$ (meaning $\varphi_\mathrm{z}=0$) is joined to one with $d>1/2$ (meaning $\varphi_\mathrm{z}=\pi$). In this section, the dislocated Su-Schrieffer-Heeger structure can also be interpreted in terms of this theory as it can be viewed as two semi-infinite materials with $d>1/2$ (meaning $\varphi_\mathrm{z}=\pi$) that are separated by a homogeneous piece of material with length $l$ (which, trivially, has $\varphi_\mathrm{z}=0$). See \emph{e.g.} \cite{ammari2020topologically, fefferman2017topologically, khanikaev2013photonic, lin2021mathematical} for more details on the notion of topological protection and developments to multi-dimensional systems.

\begin{figure}
\centering
\begin{subfigure}[b]{0.45\linewidth}
\includegraphics[width=\linewidth]{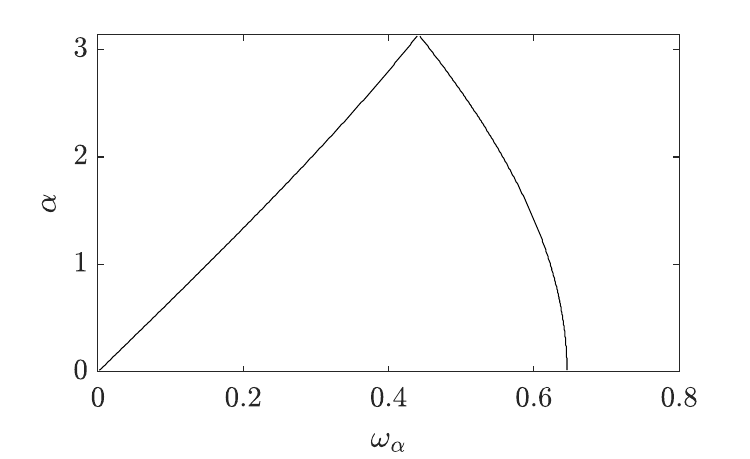}
\caption{} \label{SSH_stuck}
\end{subfigure}
\begin{subfigure}[b]{0.45\linewidth}
\includegraphics[width=\linewidth]{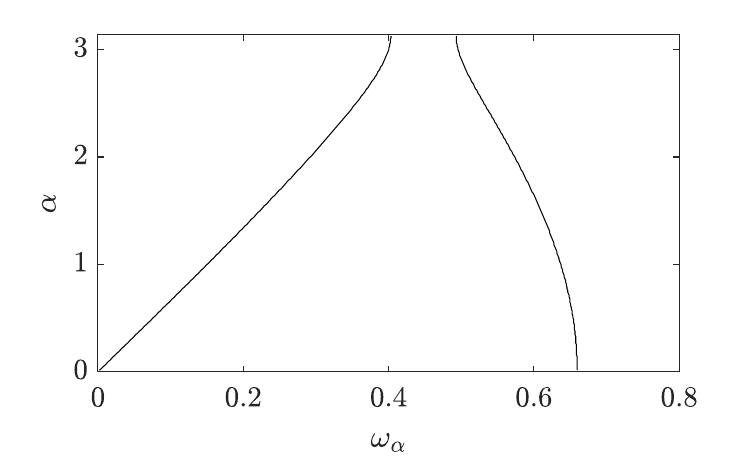}
\caption{} \label{SSH_open}
\end{subfigure}

\vspace{-0.5cm}

\caption{The spectral bands of the Su-Schrieffer-Heeger medium \eqref{eq:c_SSH}, for (a) $d=1/2$ and (b) $d\neq1/2$. If $d\neq1/2$, there is a band gap between the first two spectral bands.} \label{fig:SSHbands}
\end{figure}

\begin{figure}
\centering
\includegraphics[width=0.8\linewidth]{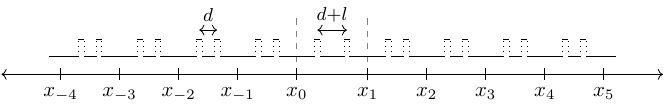}
\caption{Dislocating the Su-Schrieffer-Heeger medium yields a symmetric and piecewise constant periodic medium, as defined in \eqref{eq:c_defect}, which is locally perturbed within the defect region and supports localised eigenmodes which can be described using this theory.} \label{fig:dislocation}
\end{figure}

We introduce a defect to \eqref{eq:c_SSH} by introducing a dislocation of length $l>0$ to the centre of one of the unit cells. This gives the medium
\begin{equation} \label{eq:c_SSH_dislocated}
c_l(x) = \begin{cases}
\frac{1}{r} & x\in\scaleto{[n+\frac{1-d-a}{2},n+\frac{1-d+a}{2})\cup[n+\frac{1+d-a}{2},n+\frac{1+d+a}{2}), n\in\Z^{<0},}{0.4cm}\\
\frac{1}{r} & x\in[\frac{1-d-a}{2},\frac{1-d+a}{2}),\, x-l\in[\frac{1+d-a}{2},\frac{1+d+a}{2}),\\
\frac{1}{r} & x-l\in\scaleto{[n+\frac{1-d-a}{2},n+\frac{1-d+a}{2})\cup[n+\frac{1+d-a}{2},n+\frac{1+d+a}{2}), n\in\Z^{\geq0},}{0.38cm}\\
1 & \text{otherwise},
\end{cases}
\end{equation}
which is depicted in \Cref{fig:dislocation}. If the dislocation $d$ is not equal to an integer multiple of the periodicity $h$, then the perturbation is non-compact since the periodic material is shifted on either side to accomodate the defect. A three-dimensional version of this material was studied in \cite{ammari2020robust}, where it was shown that two localised eigenmodes are introduced with eigenfrequencies in the first band gap and these eigenfrequencies converge to a point in the middle of the band gap as the dislocation becomes arbitrarily large. Similar dislocated media have also been studied in the setting of Schr\"odinger operators \cite{drouot2020defect, gontier2020edge}.

\begin{figure}
\centering
\includegraphics[width=0.6\linewidth]{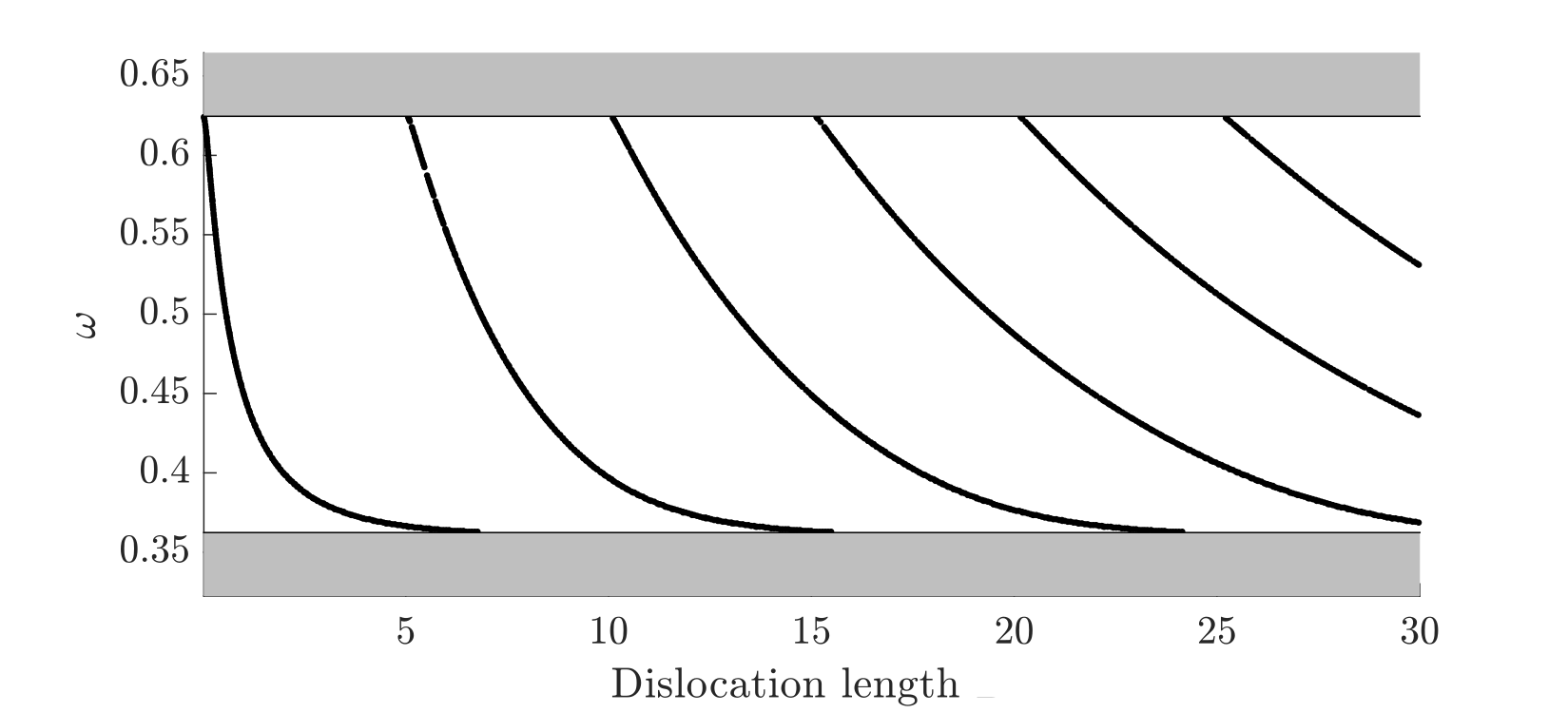}
\caption{Localised eigenmodes are created when a dislocation is introduced to the Su-Schrieffer-Heeger medium \eqref{eq:c_SSH}. The shaded regions are the spectrum of the unperturbed medium. As the dislocation length $l$ increases, each mode's eigenfrequency crosses the band gap. Here, $r=10$, $d=0.25$ and $a=0.1$.} \label{fig:dislocation_bands}
\end{figure}

When the medium is perturbed by the introduction of a dislocation, we can use \Cref{thm:localised} to find the eigenfrequencies of the localised eigenmodes. The result of solving the resulting equation is shown in \Cref{fig:dislocation_bands}. We see that successive localised eigenmodes are created, with eigenfrequencies lying within the first spectral band gap.

\begin{figure}
\centering
\begin{subfigure}[b]{0.43\linewidth}
\includegraphics[width=\linewidth]{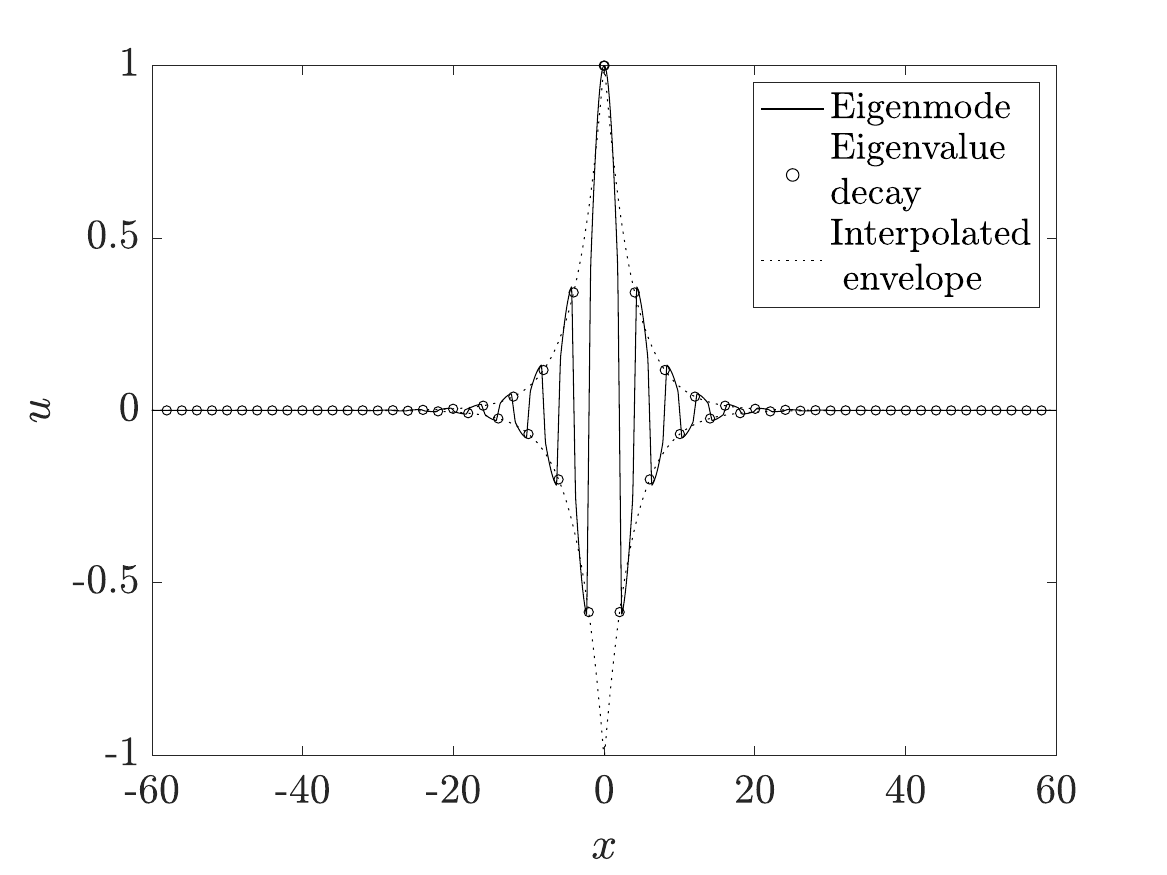}
\caption{} \label{TM_disloc1}
\end{subfigure}
\begin{subfigure}[b]{0.43\linewidth}
\includegraphics[width=\linewidth]{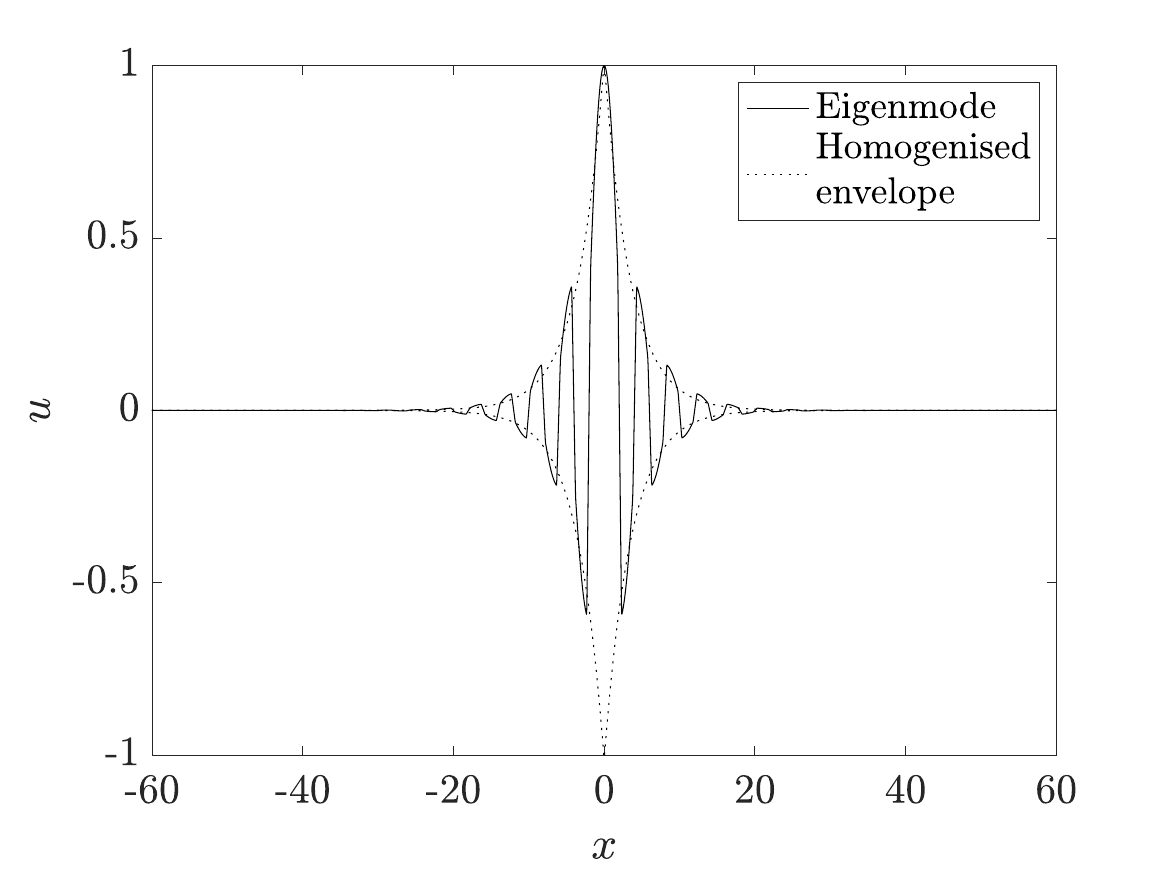}
\caption{} \label{HFH_disloc1}
\end{subfigure}

\vspace{-0.5cm}

\caption{The localised eigenmode of the dislocated Su-Schrieffer-Heeger medium \eqref{eq:c_SSH_dislocated} for dislocation length $l=0.1$. The mode has eigenfrequency $\omega=0.609$ and is shown with the transfer matrix eigenvalue envelope in (a) and the homogenised envelope in (b). The parameter values $r=10$, $d=0.25$ and $a=0.1$ are used.} \label{fig:disloc_modes_01}
\end{figure}

For small dislocation lengths $l$, we can use $\epsilon=l$ as the small parameter to use the high-frequency homogenisation results from \Cref{sec:HFH}. In particular, since the defect region is symmetric, we can use \Cref{thm:HFH_envelope} to find the homogenised envelope of the localised eigenmode. For $\epsilon=0.1$ the localised eigenmode is shown in \Cref{fig:disloc_modes_01}. Both the envelope interpolated from the transfer matrix eigenvalues and the homogenised envelope approximate the decay rate well.

To understand the nature of the curves in \Cref{fig:dislocation_bands}, it is informative to examine the localised eigenmodes that exist at larger dislocation lengths. What we find is that each successive curve of localised eigenmodes in \Cref{fig:dislocation_bands} has an additional oscillation within the defect region. At the dislocation length $l=22$, for example, the third, fourth and fifth curves persist in \Cref{fig:dislocation_bands}, so we expect the modes to have three, four and five oscillations within the defect region. This is confirmed by the plots in \Cref{fig:disloc22}. In this case, since the perturbation to the medium is very large, we do not generally expect the the homogenisation to accurately predict the decay rate (the transfer matrix eigenvalues will work regardless, as shown in \Cref{fig:disloc22}). However, the crucial distance that we require to be small in order for the high-frequency homogenisation to work is the distance of the localised frequency from the edge of the band gap. Since, in the $l=22$ case, the third eigenmode ($\omega=0.58$) has a frequency that is relatively close to the edge of the band gap (at $\omega=0.60$) the homogenised envelope still approximates the rate of decay well, as shown in \Cref{HFH_22_disloc3}. To study the modes whose eigenfrequencies are far from the edge of the band gap, a Wentzel-Kramers-Brillouin (WKB) ansatz can be used, as shown in \cite{schnitzer2017waves}.

\begin{figure}
\centering
\begin{subfigure}[b]{0.43\linewidth}
\includegraphics[width=\linewidth]{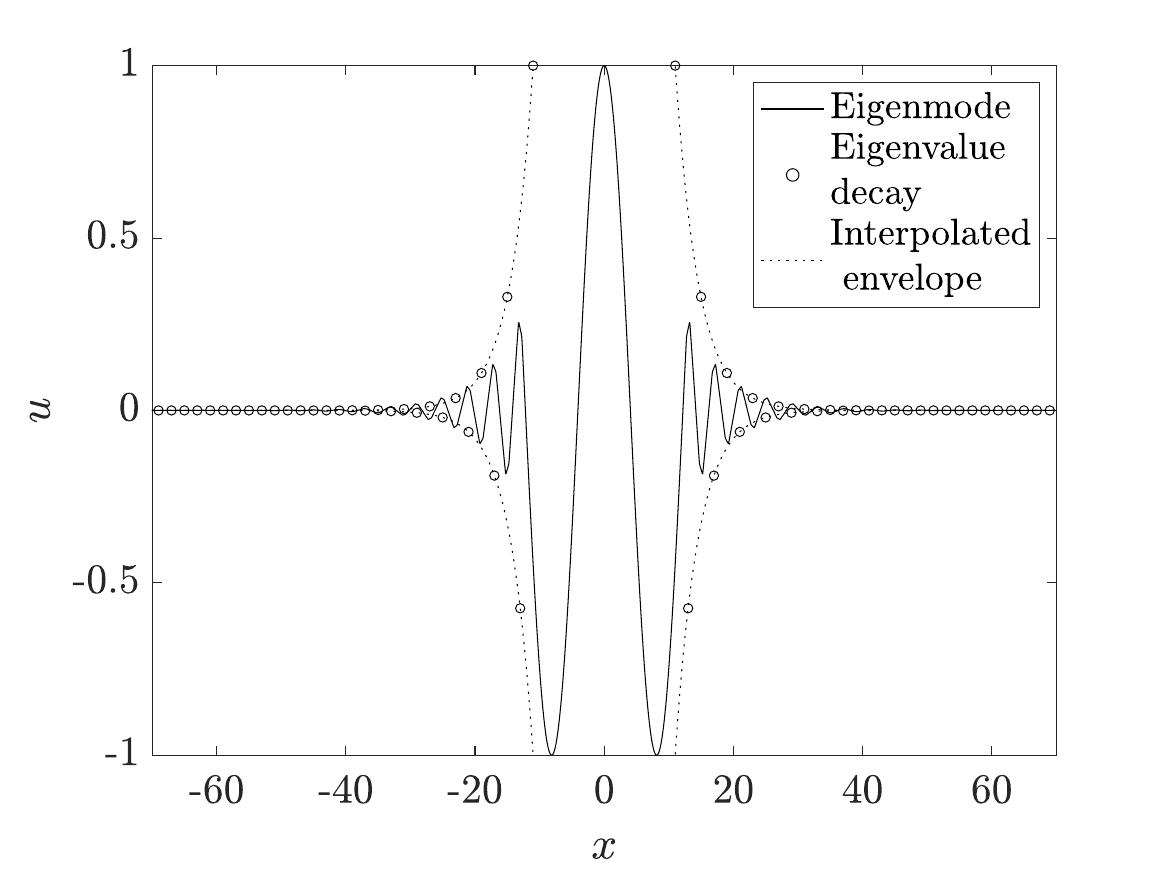}
\caption{} 
\end{subfigure}
\begin{subfigure}[b]{0.43\linewidth}
\includegraphics[width=\linewidth]{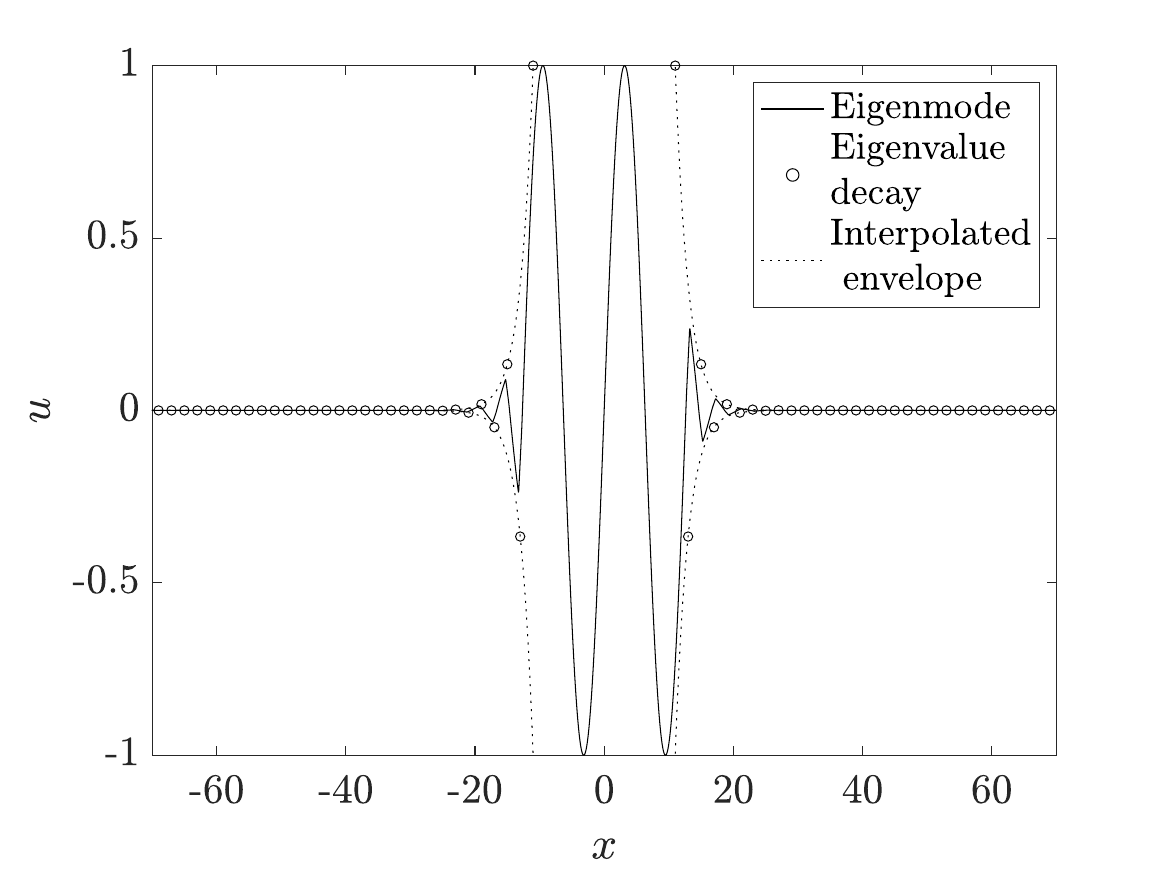}
\caption{} 
\end{subfigure}

\begin{subfigure}[b]{0.43\linewidth}
\includegraphics[width=\linewidth]{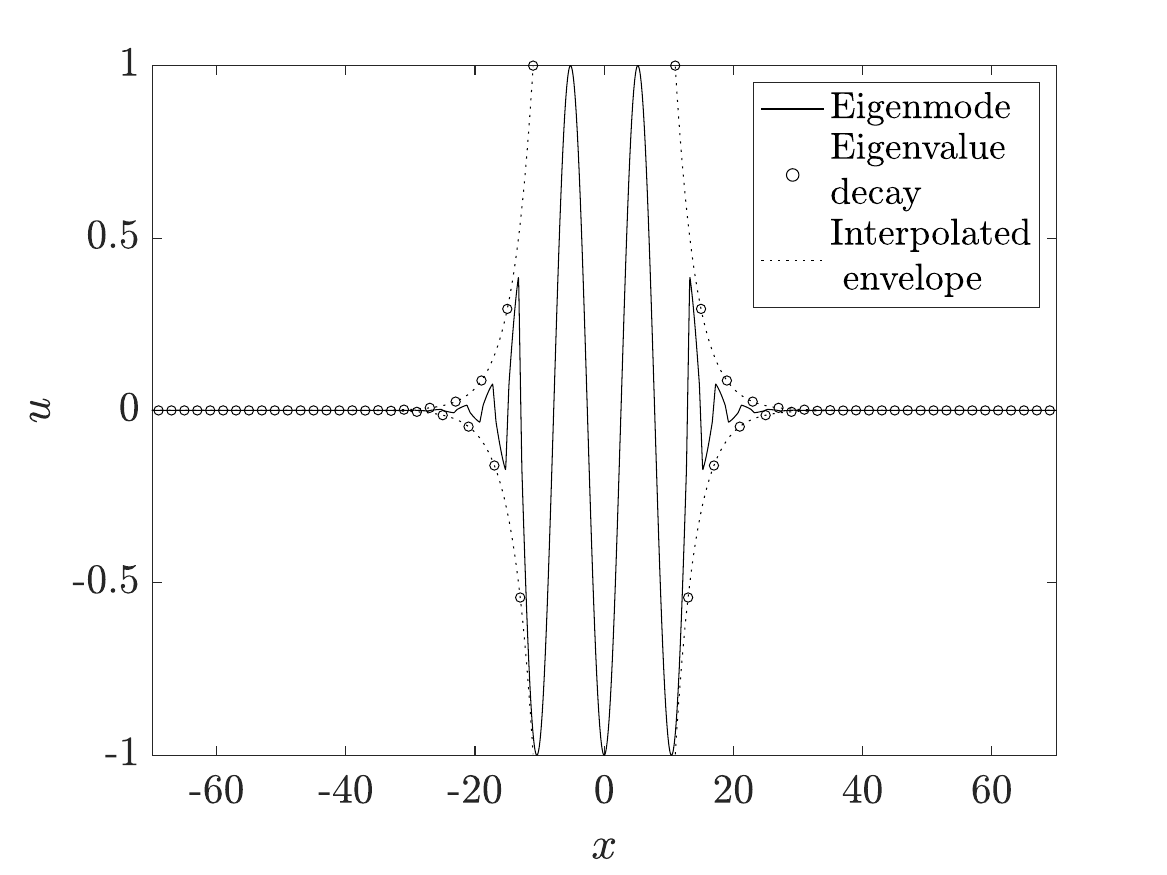}
\caption{}
\end{subfigure}
\begin{subfigure}[b]{0.43\linewidth}
\includegraphics[width=\linewidth]{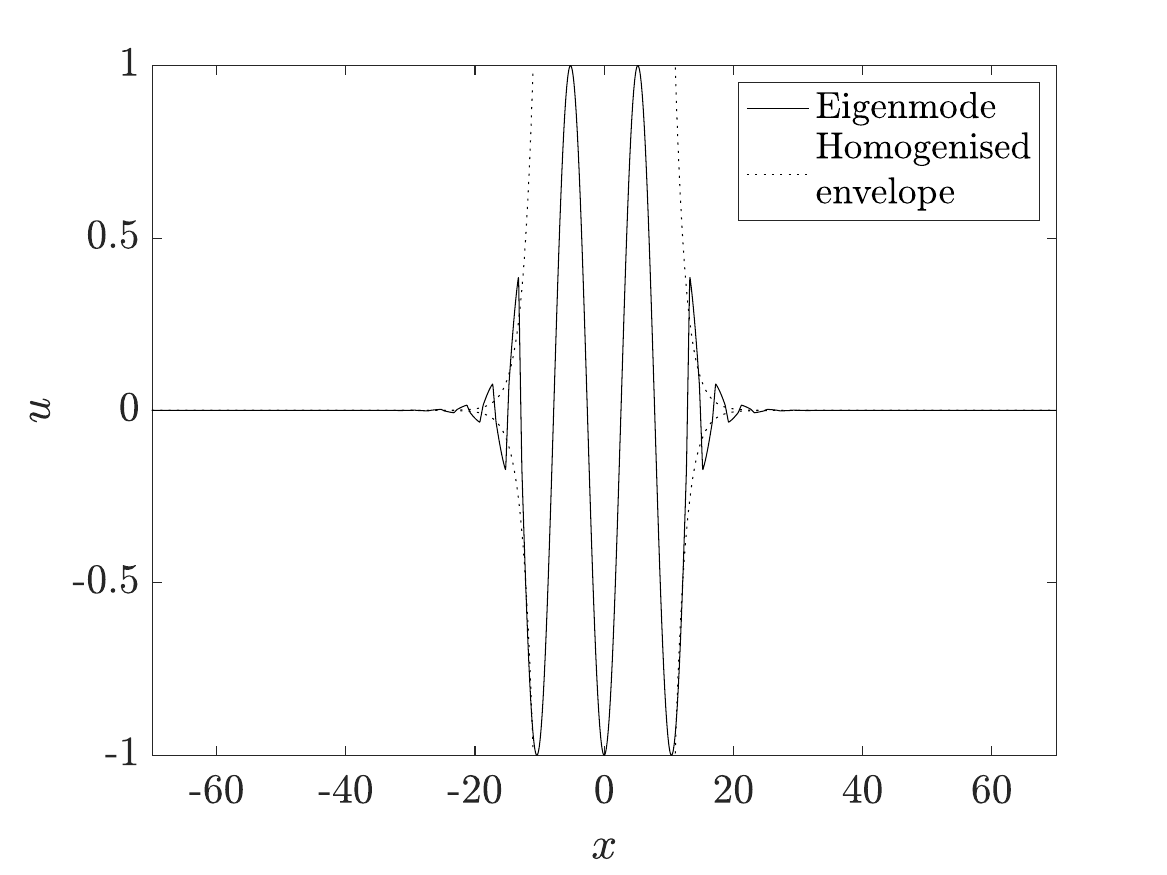}
\caption{} \label{HFH_22_disloc3}
\end{subfigure}

\vspace{-0.4cm}

\caption{The localised eigenmodes of the dislocated Su-Schrieffer-Heeger medium \eqref{eq:c_SSH_dislocated} for dislocation length $l=22$. The modes have the eigenfrequencies (a) $\omega=0.367$, (b) $\omega=0.449$ and (c) $\omega=0.577$ and are shown with the transfer matrix eigenvalue envelopes. In (d), the eigenmode with frequency $\omega=0.577$ is shown with the homogenised envelope. The parameter values $r=10$, $d=0.25$ and $a=0.1$ are used.} \label{fig:disloc22}
\end{figure}

\section{Interfacial defects in periodic media} \label{sec:half}

With some of the recent breakthroughs in topological waveguides in mind, we now explore a related setting where two different periodic media are stuck together to create an interface. We will present an approach to study the case where the two media are distinct but have the same spectral band structure. A convenient way to obtain such a material is to start with a periodic medium and make `equal but opposite' perturbations on either side of the defect. This approach was used in \cite{makwana2019tunable, makwana2018geometrically} to design tunable topological waveguides. We will achieve a similar effect by starting with a Su-Schrieffer-Heeger medium, as studied in \Cref{sec:noncompact}, with $d=1/2$ so that the first two bands stick together to close the first band gap (as shown in \Cref{fig:SSHbands}). We then perturb $d\mapsto 1/2-\epsilon$ on the left of the interface and $d\mapsto 1/2+\epsilon$ on the right of the interface, to give the medium depicted in \Cref{fig:SSH_interface_medium}, which satisfies our requirements.

\subsection{Problem setting} Suppose we have two media, labelled $A$ and $B$, which are both periodic, symmetric and which have the same associated spectra. Consider the material with a defect at $x=x_0$, which is equal to material $A$ for $x<x_0$ and equal to material $B$ for $x>x_0$. That is, if $h>0$ we have a mesh of points $\{x_n:n\in\Z\}$ which is such that
\begin{equation}
x_{n+1}-x_{n}=h \quad \text{for all } n\in\Z,
\end{equation}
and a piecewise smooth function $c:\R\to\R$ which is periodic on $\R^{>0}$ and $\R^{<0}$, in the sense that
% \begin{equation} %\label{periodicity_half}
% c(x_n+t)=c(x_m+t) \quad \text{for any } t\in[0,h) \text{ and any } n,m\in\Z^{<0},
% \end{equation}
\begin{equation} \label{periodicity_half}
c(x_n+t)=c(x_m+t) \quad \text{for any } t\in[0,h) \text{ and any }\begin{cases} n,m\in\Z^{<0}, \\ n,m\in\Z^{\geq0}. \end{cases}
\end{equation}
Additionally, $c$ is symmetric in the sense that
\begin{equation} \label{symmetry_half}
c(x_n+t)=c(x_{n+1}-t) \quad \text{for any } t\in[0,h) \text{ and any }n\in\Z.
\end{equation}
An example of the mesh  $\{x_n:n\in\Z\}$ is shown in \Cref{fig:SSH_interface_medium}, along with the exemplar medium that we will study here. This medium is a Su-Schrieffer-Heeger interface medium with $d= 1/2\pm\epsilon$ on either side of the interface, and is defined precisely as
\begin{equation} \label{eq:c_SSH_interface}
c_\epsilon(x) = \begin{cases}
\frac{1}{r} & x\in\scaleto{[\frac{2n+1-d+\epsilon-a}{2},\frac{2n+1-d+\epsilon+a}{2})\cup[\frac{2n+1+d-\epsilon-a}{2},\frac{2n+1+d-\epsilon+a}{2})}{0.38cm}, n\in\Z^{<0},\\
\frac{1}{r} & x\in\scaleto{[\frac{2n+1-d-\epsilon-a}{2},\frac{2n+1-d-\epsilon+a}{2})\cup[\frac{2n+1+d+\epsilon-a}{2},\frac{2n+1+d+\epsilon+a}{2})}{0.38cm}, n\in\Z^{\geq0},\\
1 & \text{otherwise}.
\end{cases}
\end{equation}

\begin{figure}
\centering
\includegraphics[width=0.7\linewidth]{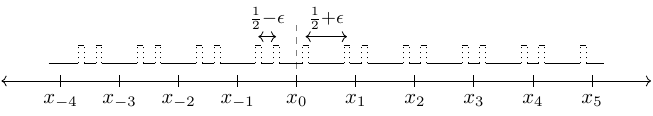}
\caption{Perturbing the Su-Schrieffer-Heeger medium on either side of an interface yields a symmetric and piecewise constant periodic medium, as defined in \eqref{eq:c_SSH_interface}, which has an semi-infinite defect and supports localised eigenmodes which can be described using this theory.} \label{fig:SSH_interface_medium}
\end{figure}

\subsection{Defect modes}
Defining the associated transfer matrices $T_A$ and $T_B$, as in \eqref{eq:transfer_left}, we see that solutions to the wave transmission problem \eqref{eq:helmholtz} with this choice of medium $c_\epsilon$ must satisfy
\begin{equation} \label{eq:transfer_half_left}
\begin{pmatrix} u(x_{n+1}) \\ u'(x_{n+1}) \end{pmatrix} = \mathbbm{1}_{\{n\geq0\}}
T_B \begin{pmatrix} u(x_n) \\ u'(x_n) \end{pmatrix} + \mathbbm{1}_{\{n<0\}}
T_A \begin{pmatrix} u(x_n) \\ u'(x_n) \end{pmatrix},
\end{equation}
for $n\in\Z$. Under our assumption of symmetry \eqref{symmetry_half}, we also have that
\begin{equation} \label{eq:transfer_half_right}
\begin{pmatrix} u(x_{n-1}) \\ u'(x_{n-1}) \end{pmatrix} = \mathbbm{1}_{\{n>0\}}
ST_BS \begin{pmatrix} u(x_n) \\ u'(x_n) \end{pmatrix} + \mathbbm{1}_{\{n\leq0\}}
ST_AS \begin{pmatrix} u(x_n) \\ u'(x_n) \end{pmatrix},
\end{equation}
for $n\in\Z$.

We can, first, prove an analogous result to \Cref{thm:localised}, determining the eigenfrequency of localised eigenmodes.

\begin{theorem} \label{thm:localised_half}
The eigenfrequency $\omega$ of a localised eigenmode of the Helmholtz problem \eqref{eq:helmholtz} posed on a medium with a semi-infinite defect must satisfy
\begin{equation*}
\begin{pmatrix} V_{21}^A(\omega) & V_{11}^A(\omega) \end{pmatrix}\begin{pmatrix} V_{11}^B(\omega) \\ V_{21}^B(\omega) \end{pmatrix} =0,
\end{equation*}
where $(V_{11}^A(\omega), V_{21}^A(\omega))^\top$ is the eigenvector of the transfer matrix $T_A(\omega)$ associated to the eigenvalue $|\lambda_1^A|<1$ and $(V_{11}^B(\omega), V_{21}^B(\omega))^\top$ is the eigenvector of $T_B(\omega)$ associated to the eigenvalue $|\lambda_1^B|<1$. Furthermore, the localised eignmode satisfies $V_{21}^Bu(x_0)=V_{11}^Bu'(x_0)$.
\end{theorem}
\begin{proof}
This is similar to the proof of \Cref{thm:localised}: use transfer matrices to propagate the solution and its derivative at $x=x_0$ in either direction and force non-decaying terms to vanish.
\end{proof}

We can, similarly, replicate \Cref{cor:eigenval_decay} in this setting, to yield an upper bound on the decay of the localised eigenmode away from the defect.

\begin{corollary} \label{cor:eigenval_decay_half}
A localised eigenmode $u$ of \eqref{eq:helmholtz}, posed on a medium with a semi-infinite defect, and its associated eigenfrequency $\omega$ must satisfy
\begin{equation*}
u(x_n)=O\Big(|\lambda_1^A(\omega)|^{-n}\Big) \quad \text{and}\quad u'(x_n)=O\Big(|\lambda_1^A(\omega)|^{-n}\Big) \quad\text{as }n\to-\infty,
\end{equation*}
% and
\begin{equation*}
u(x_n)=O\Big(|\lambda_1^B(\omega)|^{n}\Big) \quad \text{and}\quad u'(x_n)=O\Big(|\lambda_1^B(\omega)|^{n}\Big) \quad\text{as }n\to\infty,
\end{equation*}
where $\lambda_1^A(\omega)$ is the eigenvalue of $T_A(\omega)$ satisfying $|\lambda_1^A(\omega)|<1$ and $\lambda_1^B(\omega)$ is the eigenvalue of $T_B(\omega)$ satisfying $|\lambda_1^B(\omega)|<1$.
\end{corollary}

\subsection{Homogenisation of semi-infinite defect modes}

If the perturbation is small, we can use the method from \Cref{sec:HFH} to compute the homogenised envelope of the localised eigenmode. In particular, we introduce the matrix $P$ which is such that 
\begin{equation}
\epsilon P = T_B - T_A.
\end{equation}
We make some useful observations about $T_A(\omega)$, $T_B(\omega)$ and $P(\omega)$, for the specific choice of medium $c_\epsilon$ from \eqref{eq:c_SSH_interface}:
\begin{lemma} \label{lem:Tequal}
For any $\omega\in\R$, it holds that
\begin{equation*}
(T_A)_{1,1} = (T_A)_{2,2} \quad \text{and}\quad (T_B)_{1,1} = (T_B)_{2,2}.
\end{equation*}
\end{lemma}
\begin{proof}
We have that 
\begin{equation} \label{eq:TA}
T_A=T_\text{c}(1,\tfrac{1-d+\epsilon-a}{2},\omega) T_\text{c}(r,a,\omega) T_\text{c}(1,d{-}\epsilon{-}a,\omega) T_\text{c}(r,a,\omega) T_\text{c}(1,\tfrac{1-d+\epsilon-a}{2},\omega),
\end{equation}
where $T_\text{c}$ was defined in \eqref{T_piecewise}. Since $(T_\text{c})_{11}\equiv (T_\text{c})_{22}$, it follows that the symmetric product of these matrices retains this property.
\end{proof}

\begin{lemma} \label{lem:Pzero}
For any $\omega\in\R$, if $d=\tfrac{1}{2}\pm\epsilon$ on either side of the interface, then it holds that
\begin{equation*}
P_{1,1}=P_{2,2}=0.
\end{equation*}
\end{lemma}
\begin{proof}
We will show that $(T_A)_{1,1} = (T_B)_{1,1}$. That $(T_A)_{2,2} = (T_B)_{2,2}$ follows from similar calculations. Expanding the formula \eqref{eq:TA}, gives that
\begin{equation} \label{eq:TA_full}
\begin{split}
(T_A)_{11}&=\frac{1}{2r^2}\big[
2 r^2 \cos((\tfrac{1}{2}-\epsilon-a) \omega) \cos((\tfrac{1}{2}+\epsilon-a) \omega) \cos(2r a \omega )\\
&\hspace{1.2cm}- 2r^2 \sin((\tfrac{1}{2}-\epsilon-a ) \omega) \sin((\tfrac{1}{2}+\epsilon-a) \omega)\cos(ra\omega)^2\\
&\hspace{1.2cm} + \sin((\tfrac{1}{2}-\epsilon-a) \omega) \sin((\tfrac{1}{2}+\epsilon-a) \omega) \sin(ra \omega)^2 \\
&\hspace{1.2cm}+ r^4 \sin((\tfrac{1}{2}-\epsilon-a) \omega) \sin((\tfrac{1}{2}+\epsilon-a) \omega) \sin(ra \omega)^2\\
&\hspace{1.2cm} - r \cos((\tfrac{1}{2}+\epsilon-a) \omega) \sin((\tfrac{1}{2}-\epsilon-a) \omega) \sin(2 ra \omega)\\
&\hspace{1.2cm}- r \cos((\tfrac{1}{2}-\epsilon-a ) \omega) \sin((\tfrac{1}{2}+\epsilon-a) \omega) \sin(2r a \omega)\\
&\hspace{1.2cm} - r^3 \cos((\tfrac{1}{2}+\epsilon-a) \omega) \sin((\tfrac{1}{2}-\epsilon-a) \omega) \sin(2 ra \omega) \\
&\hspace{1.2cm} - r^3 \cos((\tfrac{1}{2}-\epsilon-a ) \omega) \sin((\tfrac{1}{2}+\epsilon-a) \omega) \sin(2 ra \omega)
\big].
\end{split}
\end{equation}
To find $(T_B)_{11}$ from \eqref{eq:TA_full} we must make the substitution $\epsilon\mapsto-\epsilon$. Under this alteration, the first four terms in \eqref{eq:TA_full} are unchanged while the 5\textsuperscript{th} and 6\textsuperscript{th} terms swap values, as do the 7\textsuperscript{th} and 8\textsuperscript{th}.
\end{proof}

We now proceed as in \Cref{sec:HFH}, relabelling $v=u'$ and then adding and subtracting \eqref{eq:transfer_half_left} and \eqref{eq:transfer_half_right}. We reach the two discrete equations
\begin{equation} \label{eq:discrete_half1}
\scaleto{\begin{pmatrix} u(x_{n+1}) + u(x_{n-1}) \\ v(x_{n+1})-v(x_{n-1}) \end{pmatrix} =
2T_A \begin{pmatrix} u(x_n) \\ 0 \end{pmatrix} +\mathbbm{1}_{n>0} 2\epsilon P\begin{pmatrix} u(x_n) \\ 0 \end{pmatrix}
+\delta_{n,0} \epsilon P \begin{pmatrix} u(x_n) \\ v(x_n) \end{pmatrix}}{0.81cm},
\end{equation}
and
\begin{equation} \label{eq:discrete_half2}
\scaleto{\begin{pmatrix} u(x_{n+1}) - u(x_{n-1}) \\ v(x_{n+1})+v(x_{n-1}) \end{pmatrix} =
2T_A \begin{pmatrix} 0 \\ v(x_n) \end{pmatrix} +\mathbbm{1}_{n>0} 2\epsilon P\begin{pmatrix} 0 \\ v(x_n)  \end{pmatrix}
+\delta_{n,0} \epsilon P \begin{pmatrix} u(x_n) \\ v(x_n) \end{pmatrix}}{0.81cm}.
\end{equation}
Using the same high-frequency homogenisation ansatz as in \eqref{hfh_ansatz1} and \eqref{hfh_ansatz2}, we find that the leading-order terms are given in terms of the long-scale variable $\eta=\epsilon n$ as $u_0(\eta,0)=f(\eta)$ and $v_0(\eta,0)=g(\eta)$, where $f$ and $g$ are solutions of
\begin{equation} \label{eq:general_f_half}
0=\ddt{f}{\eta}-\T_f^2 f+\delta(\eta) (P_0)_{12} g(0)
\quad\text{and}\quad
0=\ddt{g}{\eta}-\T_f^2 g+\delta(\eta) (P_0)_{21} f(0) ,
\end{equation}
which decay as $\eta\to\pm\infty$. Here, we have defined $\T_f:=\sqrt{-2(T_2)_{11}}$ and have used \Cref{lem:Tequal} to deduce that $\T_f=\sqrt{-2(T_2)_{22}}$ also. \Cref{lem:Pzero} has been crucial to guarantee that several terms from \eqref{eq:discrete_half1} and \eqref{eq:discrete_half2} vanish. We can solve the coupled equations \eqref{eq:general_f_half} using \Cref{lem:f_solution} to deduce the following result:

\begin{theorem} \label{thm:HFH_envelope_half}
Under the homogenisation ansatz \eqref{hfh_ansatz2}, a localised eigenmode of \eqref{eq:helmholtz} posed on the medium $c_\epsilon$, which has a semi-infinite defect, must satisfy
\begin{equation*}
|u(x_n)|=\exp\left(-\tfrac{1}{2}\sqrt{\left((T_B)_{12}-(T_A)_{12}\right)\left((T_B)_{21}-(T_A)_{21}\right)}|n|\right)+O(\epsilon) \qquad \text{for } n\in\Z,
\end{equation*}
when normalised such that $\max_n|u(x_n)|=1$.
\end{theorem}
\begin{proof}
Using \Cref{lem:f_solution}, we can solve \eqref{eq:general_f_half} to find that 
\begin{equation}
    f(\eta) = \frac{(P_0)_{12}(P_0)_{21}g(0)}{4\T_f^2} e^{-\T_f|\eta|}
    \quad\text{and}\quad
    g(\eta) = \frac{(P_0)_{12}(P_0)_{21} f(0)}{4\T_f^2} e^{-\T_f|\eta|}.
\end{equation}
For consistency, we must have that $\T_f=\tfrac{1}{2}\sqrt{(P_0)_{12}(P_0)_{21}}$. Swapping coordinates to the discrete short-scale variable and substituting the definition of $P$ gives the result.
\end{proof}

\subsection{Example: Su-Schrieffer-Heeger interface modes}

\begin{figure}
\centering
\begin{subfigure}[b]{0.45\linewidth}
\includegraphics[width=\linewidth]{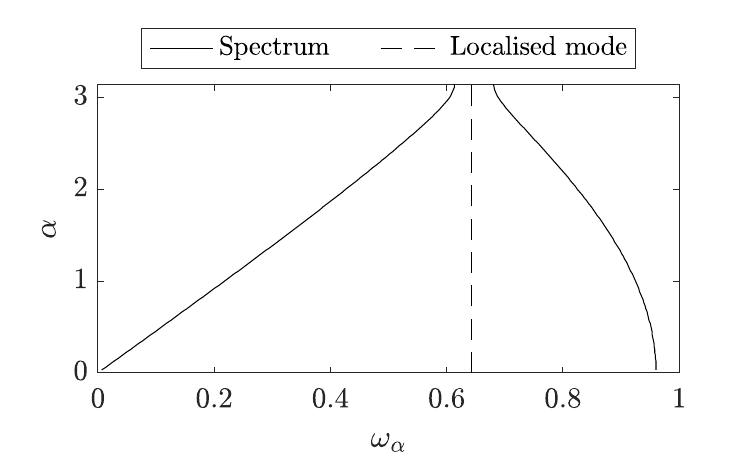}
\caption{} \label{intspec}
\end{subfigure}

\begin{subfigure}[b]{0.43\linewidth}
\includegraphics[width=\linewidth]{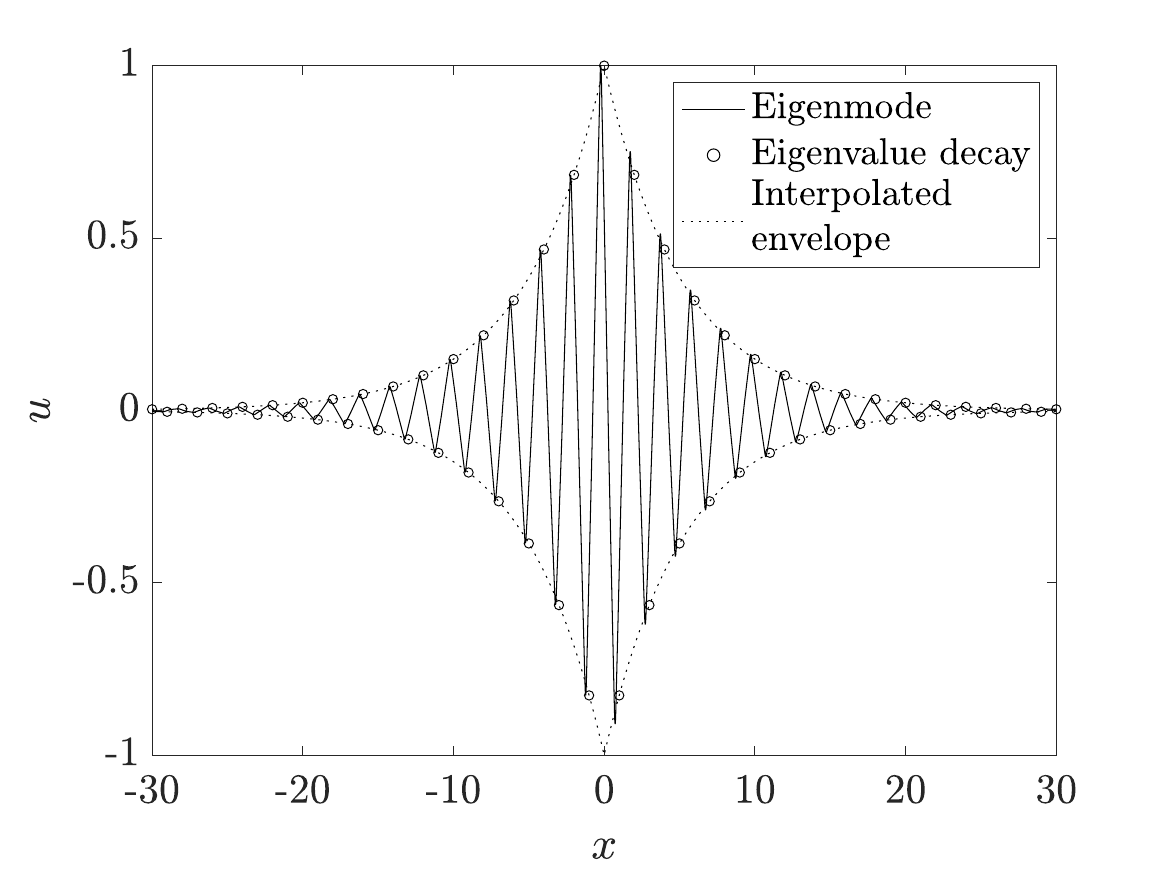}
\caption{} \label{TM_int}
\end{subfigure}
\begin{subfigure}[b]{0.43\linewidth}
\includegraphics[width=\linewidth]{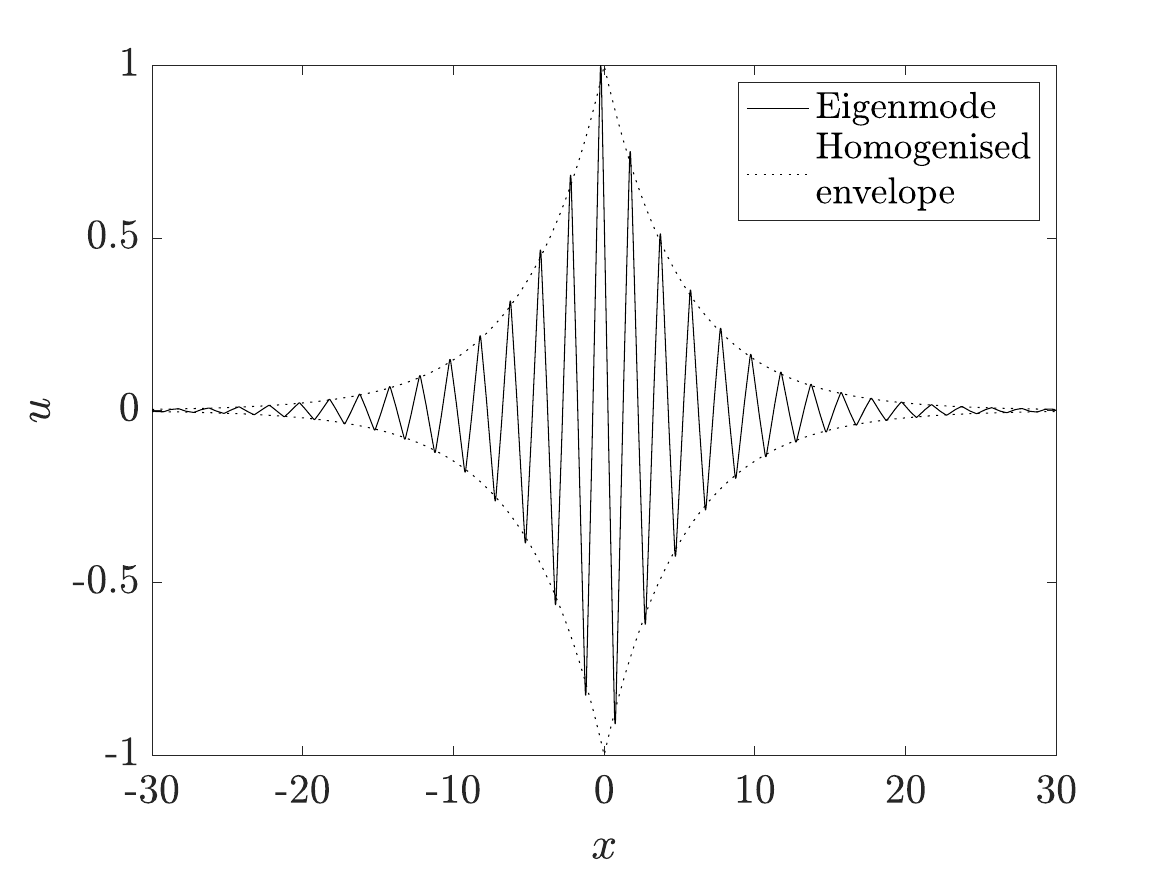}
\caption{} \label{HFH_int}
\end{subfigure}

\vspace{-0.5cm}

\caption{The localised eigenmodes of the Su-Schrieffer-Heeger interface medium \eqref{eq:c_SSH_interface} for $\epsilon=0.05$. The (coincident) spectrum of the unperturbed periodic media is shown in (a), along with the frequency of the localised eigenmode. The localised eigenmode is shown in (b) and (c) with the envelope functions computed using transfer matrix eigenvalues and homogenisation, respectively. Here, $r=10$ and $a=0.1$.} \label{fig:interfaceModes}
\end{figure}

The theory developed for semi-infinite perturbations in this section has been specifically developed for the case of the Su-Schrieffer-Heeger interface structure $c_\epsilon$ defined in \eqref{eq:c_SSH_interface}; in particular, \Cref{lem:Tequal} and \Cref{lem:Pzero} hold specifically for this medium. \Cref{thm:localised_half} can be used to find the eigenfrequency of the localised eigenmode that exists within the band gap that is opened for $\epsilon>0$. This is shown in \Cref{intspec} for the case of $\epsilon=0.05$. We can, subsequently, use \Cref{cor:eigenval_decay_half} and \Cref{thm:HFH_envelope_half} to estimate the rate at which the localised eigenmode decays away from the interface. This is shown in \Cref{TM_int,HFH_int}.

\section{Application to design problems} \label{sec:rainbow}

The value of the results developed in this work, compared to existing approaches including direct numerical simulation for example, is that they describe the crucial properties of localised eigenmodes very concisely. This means they can be used to very efficiently tune and optimise materials for applications. As a demonstrative example, we would like to design a simple rainbow filtering device that localises specific frequencies at different positions, thereby separating a wave into its corresponding frequency components. These devices have been used for many significant applications including energy harvesting \cite{chaplain2020topological}, machine hearing \cite{ammari2020mimicking, davies2021robustness} and metamaterial spectrometry \cite{tsakmakidis2007trapped}.

The periodic medium that was studied in \Cref{sec:local} has a band gap for $\omega\in(0.173,0.316)$ (the spectral band structure is shown in \Cref{PDspec}). Suppose that we wish to design a device that separates the frequencies $\omega=0.2$ and $\omega=0.3$. We can use \Cref{thm:localised} to find the perturbed material parameters that give localised eigenmodes with these eigenfrequencies. We find that $R=11.54$ gives a localised eigenmode with frequency $\omega=0.2$ and $R=6.61$ gives $\omega=0.3$ (the unperturbed periodic structure has $r=10$). We can then create a simple rainbow filtering device by introducing two defects, with $R=11.54$ and $R=6.61$. In \Cref{fig:rainbow} we show the two localised eigenmodes for this rainbow filtering device (in this case, there are ten unperturbed unit cells between the two defects). In spite of the fact that the appropriate defects were calculated using theory for just a single defect, the eigenfrequencies are not changed significantly by the interactions between the two defects in this case.

 \begin{figure}
\centering
\includegraphics[width=0.5\linewidth]{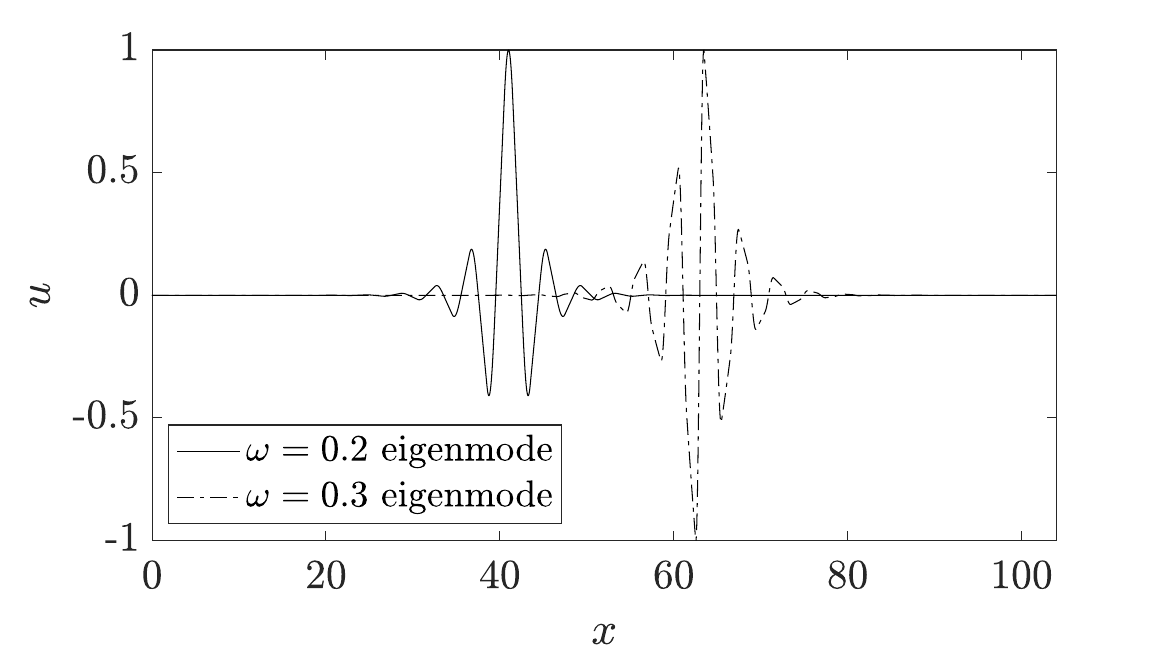}
\caption{The results of this work can be used to design a simple rainbow filtering device that has multiple defects to separate the corresponding frequencies. Here, the periodic medium has two defects and two corresponding localised eigenmodes with eigenfrequencies $\omega=0.300$ and $\omega=0.200$.} \label{fig:rainbow}
\end{figure}

The Su-Schrieffer-Heeger periodic material defined in \eqref{eq:c_SSH} could also be used to design tunable rainbow filtering devices. In this case, the localised eigenmodes would be topologically protected. In particular, we saw in \Cref{sec:noncompact} that when a dislocation is introduced to the medium, a localised eigenmode is created and its frequency crosses the band gap as the dislocation length increases. This means that for any frequency within the band gap we can find a localised eigenmode having that eigenfrequency.

\section{Conclusions}
\label{sec:conclusions}

We have developed an approach that takes advantage of the convenient theory of transfer matrices and extends previous work on high-frequency homogenisation of discrete systems with defects \cite{makwana2013localised, craster2010lattice} to continuous differential problems. This also extends previous homogenisation results for homogeneous background media \cite{marigo2017effective} and builds on work characterising the existence of topologically protected edge modes \cite{fefferman2017topologically, lin2021mathematical, gontier2020edge} by providing a means to quantify their most important properties. Given the huge importance of topologically protected states in wave physics we chose to study localised defect modes in media based on the Su-Schrieffer-Heeger model, which supports one-dimensional topologically protected eigenmodes. Since high-frequency homogenisation has been used previously to model multi-dimensional periodic media \cite{craster2010high, craster2010lattice}, this work offers an analytic approach to studying the rich variety of topological waveguides that have been developed \cite{khanikaev2013photonic, makwana2019tunable, makwana2018geometrically}. The concise formulas developed in this work both deepen understanding and facilitate efficient design optimisation, without the need for expensive simulations (as was done for the rainbow filtering device in \Cref{sec:rainbow}).

An interesting direction for future investigation would be to explore the creation of localised eigenmodes due to defects in non-periodic structures. For example, it has been shown that introducing symmetries to quasicrystalline materials can create localised eigenmodes \cite{apigo2019observation, davies2022symmetry, marti2021edge}. Some crude estimates for the decay rates of these modes were computed by the present authors in \cite{davies2022symmetry}. However, the classical homogenisation theory that has been developed for quasicrystals, \emph{e.g.} by \cite{bouchitte2010homogenization, ganesh2022bloch, wellander2018two}, is yet to be extended away from low frequencies.

% These results yield explicit formulas which can be used to tune the designs of materials very efficiently, according to custom specifications (\emph{e.g.} the rainbow filtering device in \Cref{sec:rainbow}).

% We have developed an approach that extends previous work on high-frequency homogenisation of discrete systems with defects \cite{makwana2013localised, craster2010lattice} to continuous differential problems. This also extends previous homogenisation results for homogeneous background media \cite{marigo2017effective} and builds on work characterising the existence of topologically protected edge modes \cite{fefferman2017topologically, lin2021mathematical} by providing a means to quantify their most important properties.

\section*{Acknowledgments}
BD was funded by the H2020 FETOpen project BOHEME under grant agreement No.~863179. The code used to produce the numerical examples presented in this article is available at \href{https://doi.org/10.5281/zenodo.5960015}{https://doi.org/10.5281/zenodo.5960015}.

\bibliographystyle{siamplain}
\bibliography{references}
\end{document}